\documentclass{article}
\usepackage{PRIMEarxiv}
\usepackage{comment}
\usepackage{hyperref}
\usepackage{amstext,amsmath,amssymb,amsbsy,bm,amsthm}
\usepackage{authblk}
\usepackage{booktabs}
\usepackage{graphicx}
\usepackage{xcolor}
\usepackage{neuralnetwork}
\usepackage{algorithm}
\usepackage{algpseudocode}
\usepackage{makecell}
\usetikzlibrary{fit,positioning}
\usepackage{subcaption}
\usepackage{cancel}

\usepackage{placeins}
\theoremstyle{thmstyleone}%
\newtheorem{theorem}{Theorem}%  meant for continuous numbers

\newtheorem{proposition}[theorem]{Proposition}% 

\newtheorem{definition}{Definition}[section]
\newtheorem{remark}[definition]{Remark}

\title{HEATNETs: Explainable Random Feature Neural Networks  for High-Dimensional Parabolic PDEs}

\author[1]{Kyriakos Georgiou}
\author[2]{Gianluca Fabiani}
\author[3,\thanks{{Corresponding author: \texttt{constantinos.siettos@unina.it}}} 
]{Constantinos Siettos}
\author[4,\thanks{{Corresponding author: \texttt{ayannaco@aueb.gr}}} 
]{Athanasios N.\ Yannacopoulos}

\affil[1]{Department of Electrical Engineering and Information Technologies, \textit{University of Naples “Federico II”}, Naples, Italy }
\affil[2]{Modelling and Engineering Risk and Complexity, Scuola Superiore Meridionale, Italy; Currently: Hopkins Extreme Materials Institute and Dept. of Chemical and Biomolecular Engineering, Johns Hopkins University, Baltimore, USA }
\affil[3]{Dipartimento di Matematica e Applicazioni ``Renato Caccioppoli'', \textit{University of Naples “Federico II”}, Naples, Italy }
\affil[4]{Department of Statistics and Stochastic Modelling and Applications Laboratory, \textit{Athens University of Economics and Business}, Athens, Greece }
\begin{document}
\maketitle
\begin{abstract}
We deal with the solution of the forward problem for high-dimensional parabolic PDEs with random feature/projection neural networks (RFNNs). We first prove that there exists a single-hidden layer neural network with randomized heat-kernels arising from the fundamental solution (Green’s functions) of the heat operator, that we call HEATNET, that provides an unbiased universal approximator to the solution of parabolic PDEs in arbitrary (high) dimensions, with the rate of convergence being analogous to the $\mathcal{O}(N^{-1/2})$, where $N$ is the size of HEATNET.  
Thus, HEATNETs are explainable schemes, based on the analytical framework of parabolic PDEs, exploiting insights from physics-informed neural networks aided by numerical and functional analysis, and the structure of the corresponding solution operators.   
Importantly, we show how HEATNETs can be scaled up for the efficient numerical solution of arbitrary high-dimensional parabolic PDEs using suitable transformations and importance Monte Carlo sampling of the integral representation of the solution, in order to deal with the singularities of the heat kernel around the collocation points. We evaluate the performance of HEATNETs through benchmark linear parabolic problems up to 2,000 dimensions. We show that HEATNETs result in remarkable accuracy with the order of the approximation error ranging from $1.0$E$-05$ to $1.0$E$-07$ for problems up to 500 dimensions, and of the order of $1.0$E$-04$ to $1.0$E$-03$ for 1,000 to 2,000 dimensions, with a relatively low number (up to 15,000) of features.

\end{abstract}

{\bf Keywords}: 
\texttt High-dimensional parabolic PDEs; Explainable Machine Learning; Random Feature Neural Networks; Heat Kernels; Importance Sampling; Universal Approximation Theorem  \\
%\MSC[2010] 00-01\sep  99-00

%\end{frontmatter}

\section{Introduction}
The origins for the use of machine learning for the numerical solution of differential equations in the form of ordinary (ODEs) and partial differential equations (PDEs) can be traced back in the early '90s. Lee and Kang~\cite{lee1990neural} used a modified Hopfield Neural Network to solve a first-order nonlinear ODE. Meade and Fernadez \cite{meade1994numerical} used Feedforward Neural Networks (FNN) for the solution of linear ODEs, which training, i.e. the estimation of the weights is based on the Galerkin weighted-residuals method. Lagaris et al. \cite{lagaris1998artificial} used feedforward neural networks (FNN) to solve, relatively simple linear and nonlinear differential equations including initial and boundary value problems and steady-state solutions of PDEs. Gerstberger and Rentrop \cite{gerstberger1997feedforward}, used FNN to implement the implicit Euler scheme for the solution of stiff ODEs and Differential-Algebraic Equations (DAEs). Gonzalez-Garcia et al. \cite{gonzalez1998identification} addressed a multilayer neural network architecture implementing the 4-th order Runge-Kutta scheme for the time integration of nonlinear PDEs learned from data. The performance of the scheme was validated considering the Kuramoto–Sivashinsky equation.  \par
More recently, theoretical and computational advances have boosted research activity allowing further developments for the solution of the forward problem of differential equations with machine learning. Efforts have been mainly focused on the use of Physics-Informed Neural Networks (PINNs) \cite{raissi2019physics,karniadakis2021physics}: deep neural networks \cite{sirignano2018dgm,han2018solving,raissi2019physics,zhang2020learning,lu2021deepxde,tang2021exploratory,de2022error,ma2023bi,hu2024tackling,lee2025fast,kavousanakis2025going} which directly parameterize the solution function and train via PDE residuals, Gaussian processes \cite{raissi2017inferring,raissi2018numerical,chen2021solving} which offer a probabilistic solution with uncertainty estimates, random feature/projection shallow neural networks (RFNN) \cite{DWIVEDI202096,calabro2021extreme,fabiani2021numerical,dong2021local,dong2021modified,galaris2022numerical,fabiani2023parsimonious,bolager2023sampling,wang2024extreme,datar2024solving,fabiani2025random}, which reduce the training process into the solution of a regularized (linear) least-squares problem through the theory of random embeddings. Neural operators (NOs) \cite{li2020fourier,lu2021learning,wang2021learning,fang2024learning,fabiani2025randonets,fabiani2025enabling}, which learn the mapping from input functions (e.g. initial or boundary data, coefficients) to the solution function, and more recently Fredholm neural networks \cite{georgiou2025fredholm1}, which emulate fixed point algorithms as DNNs with predetermined by the iterative method, weights and biases, to solve elliptic PDEs have been also used. 

\subsection{Literature review and state-of-the-art for high-dimensional PDEs}
Most machine-learning methods for PDEs focus on low-dimensional problems. They are rarely compared directly with classical methods such as finite elements, finite differences, or finite volumes, which are usually more accurate and efficient. While, some (few) works have shown that well-designed shallow networks, like random-feature neural networks, can compete with traditional methods for ODEs and PDEs \cite{fabiani2021numerical,fabiani2023parsimonious,grossmann2024can,fabiani2025random}, the main challenge remains solving high-dimensional spatial PDEs, where standard numerical methods are no longer practical. Towards this aim, Raissi et al. \cite{raissi2017inferring} used Gaussian Processes for the numerical solution of both stationary and time-dependent linear integro-differential operators from multi-fidelity data including the Possion equation in 10 dimensions.  
Han et al. \cite{han2018solving} used DNNs to solve high-dimensional nonlinear parabolic PDEs including the Black–Scholes, the Hamilton–Jacobi–Bellman and the Allen–Cahn equation in 100 dimensions. Wei et al. \cite{wei2018machine} used single-layer FNNs to solve modified high-dimensional diffusion equations up to 100 dimensions. 
Sirignano and Spiliopoulos \cite{sirignano2018dgm}, introduced the Deep Galerkin Method (DGM)
for the solution of high-dimensional free boundary PDEs up to 200 dimensions including the Hamilton–Jacobi–Bellman PDE. The training of the DNNs was distributed across several GPU nodes. A theorem regarding the approximation power of DNNs for quasilinear PDEs is also given. 
Chan et al. \cite{chan2019machine} used DNNs to solve time-dependent semi-linear PDEs up to 100 dimensions.  
Hu et al. \cite{hu2024tackling}, scaled up PINNs using a stochastic dimension gradient descent method for their training to solve arbitrary high-dimensional PDEs,  including the Hamilton–Jacobi-Bellman and the Schrödinger equations in tens of thousands of dimensions. 
Finally, in \cite{wang2024extreme}, Qang and Dong, proposed a method based on randomized neural networks (ELMs)-whose training is achieved via least-squares and an approximate variant of the theory of functional connections to solve high-dimensional PDEs including the heat, Korteweg-de Vries (KdV), and the advection diffusion equation up to 15 dimensions.\par

Each approach has its own pros and cons: for example, PINNs  and DNNs \cite{wang2021understanding} although highly expressive and scalable, can be hard to train (even if training is distributed across GPUs), they lack interpretability and, in general, a rigorous theoretical framework for uncertainty quantification. Gaussian Process models offer a rigorous theoretical Bayesian framework with built-in uncertainty quantification, but scale poorly with large datasets or high dimensions and also require GPUs for training. RFNNs offer the advantage of reducing training to a (linear) least-squares problem with strong theoretical guarantees, the arsenal of parallelizable and matrix-free linear algebra algorithms in the Krylov subspace, and connections to kernel methods with possible uncertainty quantification. However, they suffer from the curse of dimensionality, since training essentially reduces to tuning shape parameters and sampling effectively in a high-dimensional space. Moreover, they are memory-intensive, as their training relies on solving large-scale least-squares problems. Therefore, they have not yet been implemented for the solution of PDEs in very-high dimensions (e.g., in \cite{wang2024extreme}, RFNNs have been used to solve PDEs up to 15 dimensions). 

\subsection{Novelty and contribution of the proposed method}
Here, in order to mitigate the limitations associated with RFNNs,  we propose HEATNETs: ``functional analysis-informed'' RFNNs with randomized heat-kernels arising from the fundamental solution (Green’s functions) of the heat operator, for the numerical solution of time-dependent parabolic PDEs in arbitrary high dimensions. The proposed scheme is informed by the analytical theory of parabolic PDEs, rooted in numerical and functional analysis and the properties of their solution operators. It builds on our previous works on random feature/projection neural networks for time-dependent differential equations \cite{fabiani2021numerical,galaris2022numerical,patsatzis2024slow,fabiani2025randonets,fabiani2025random2} and
recent theoretical insights on learning Green's functions associated with parabolic
PDEs via randomized SVD which allows to learn Hilbert Schmidt operators, to construct a low-rank approximant using random functions drawn from a Gaussian process \cite{boulle2022learning}. Specifically, we approximate the solution to parabolic PDEs in arbitrary high-dimensions by computing the probabilistic expectation of the transition density of Brownian motion-corresponding to the Green's functions of the diffusion operator-biased by the accumulated effect of the source term, resulting in an integral (mild) representation of the solution. This representation is the basis of the HEATNET construction: we consider a RFNN which simulates the Monte Carlo approximation of the mild solution, by using as features either the Gaussian basis functions or terms from an importance sampling scheme, in order to deal with the curse of dimensionality. 

We prove that the HEATNETs are unbiased universal approximators of the solution of parabolic PDEs in arbitrary high dimensions, and we demonstrate that their convergence rate scales as $\mathcal{O}(N^{-1/2})$, where $N$ is the number of neurons in the hidden layer.  
\par Our theoretical results provide a direct link between the aforementioned representation of the solution to the PDE and the custom RFNN. We then show, how we can train the HEATNET model using the physics informed machine learning framework \cite{raissi2019physics,karniadakis2021physics}, that results in a regularized least-squares problem.  We show that, by combining the mathematically-grounded custom model with the standard linear algebra/ numerical analysis arsenal, we can obtain the trained model which requires significantly less samples than the standard naive Monte Carlo approximation, and simultaneously is a universal approximator. 
Most importantly, by taking advantage of importance sampling, we show that the proposed scheme provides an appealing and explainable approach for the solution of parabolic PDEs in high dimensions. For our illustrations, we consider examples with separable and non-separable solutions, up to 2,000 dimensions, and give a detailed analysis of the HEATNETs performance in such problems with examples that showcase the models' accuracy, generalization ability, and the effect of the number of features.  
 
\section{Problem Statement and preliminaries}
We will propose a custom random neural network with random features, designed to solve the time-dependent parabolic PDEs defined by:
\begin{equation}
\begin{split}
    u_t(t,x)&=\mathcal{F}(u,Du, D^2 u), \quad (t,x) \in [t_0,t_f] \times {\cal D}\\
    \mathcal{B}(u)&=g(t,x), \quad (t,x) \in [t_0,t_f] \times\partial {\cal D}\\
    u(t_0,x)&=f(x), \quad x \in {\cal D},
\end{split}
\label{eq:problem}
\end{equation}
where  ${\cal D} \subset {\mathbb R}^{d}$ is a suitable spatial domain (the case  ${\cal D}={\mathbb R}^{d}$ is also allowed), $\mathcal{F}$ is a nonlinear elliptic differential operator, $\mathcal{B}u=g$ are the boundary conditions and $f$ is the initial condition. In the case where ${\cal D}={\mathbb R}^{d}$, the solution of the problem requires appropriate choice of the boundary conditions in terms of suitable decay conditions of the solution at infinity. \\
In what follows, we briefly revise key concepts on the representation of solutions of problems in the general form \eqref{eq:problem}, which will form the basis for our proposed randomized shallow neural network architecture. 

\subsection{Preliminaries: Potential representation and mild solutions}

We will focus our attention on operators ${\cal F}$ in the following form:
\begin{equation}
\begin{aligned}
{\cal F}(u, D u, D^2 u)= div( A(t,x,u) Du ) + b(t,x,u) \cdot D u + c(t,x,u) u +F(t,x, u), 
\end{aligned}
\end{equation}
where $A \in {\mathbb R}^{d \times d}$ is a matrix valued function, $b \in {\mathbb R}^{d}$
is a vector valued function and $c$ and $f$ are scalar valued functions. In the case where $A,b, c, f$ are independent of $u$, the problem reduces to a linear parabolic equation. It is also feasible to let $A, b, c, f$ depend also on $Du$ in which case the equation is a quasilinear parabolic equation. An important assumption is that the matrix $A$ satisfies the ellipticity property (\cite{boulle2022learning}):
\begin{equation}
    \begin{aligned}
 c | \xi |^2 \le A \xi \cdot \xi  \le C | \xi|^2, \,\,\, \forall \,\, \xi \in {\mathbb R}^{d},
    \end{aligned}
    \end{equation}
for some $c, C \ge 0$. In the case where $A$ depends on $t, x$ or even $u$, this condition is assumed to hold for almost all $t, x$ and possible real values of $u$.

In what follows, we first consider the linear problem
\begin{equation}\label{11111}
\begin{aligned}
&u_t=L u + F, \, \,
&L u = \sum_{i,j=1}^{n} D_{i}(a_{ij} D_{j}u) + \sum_{i=1}^{n} b_i D_i u + c u,
\end{aligned}
\end{equation}
where $A=(a_{ij}), b=(b_i), c$ are possibly functions of $t,x$.
We will also define the operator ${\cal L}$ by
\begin{equation}
    {\cal L} f:=f_t-Lf,
\end{equation}
for any function $f$, so that problem \eqref{11111} reduces to the abstract operator equation
\begin{equation}
{\cal L}u=F,
\end{equation}
as well as the compact notation $X=(t,x)$ and $Z=(s,z)$.

A Green's function for the operator ${\cal L}$ is a solution of the problem
\begin{equation}\label{22222}
\begin{aligned}
{\cal L}_{Z}^{*}G(X,Z)=\delta(X-Z), \, \,
\lim_{s \to t^{-}}G(X,Z)=\delta(x-z) ,
\end{aligned}
\end{equation}
\color{black}
where ${\cal L}_{Z}^{*}$ is the adjoint operator of ${\cal L}$, with the subscript $Z$ denoting the fact that this operator is assumed to act on functions of $Z=(s,z)$. We will also use the alternative notation 
\begin{equation}
G(X,Z)=G(t,x ; s, z),    
\end{equation}
for the Green's function when needed.
Both the system \eqref{11111} as well as the definition of the Green's function \eqref{22222} can be complemented with boundary conditions.

An important use of the Green's function is in the integral representation of the solution of \eqref{11111} in terms of (see also in \cite{boulle2022learning}):
\begin{equation}\label{33333}
\begin{aligned}
u(t,x)=\int_{{\cal D}} G(t, x ; 0, z)u(0, z)dz +\int_{0}^{T} \int_{ D} G(t,x ; s, z) F(s, z) dz ds =: ({\cal G} \star F)(t,x).
\end{aligned}
\end{equation}
Representation \eqref{33333} will form the basis of our approach.

The existence of Green's functions for parabolic equations as well as their properties is a well and extensively studied problem, up to date.  Interesting technical issues arise from the fact that the Green's function is a measure valued solution of the parabolic PDE. Various techniques have been proposed concerning the construction of Green's functions analytically and numerically. 

Moreover, for certain classes of important operators the Green's function is known in closed form.
A particularly important special case is the case where $L=D \Delta$, with $D>0$ and $\Delta$ the Laplacian, defined on the whole space ${\mathbb R}^{d}$, in which case:
\begin{equation}\label{30-10-2025}
G(t,x;s,z):= (2 \pi D)^{-d/2}  (t-s)^{-d/2} \exp\left( -\frac{\| x- z\|^2}{4 D (t-s)} \right) {\mathbf 1}_{[0, t]}(s),
\end{equation}
where ${\mathbf 1}_{[0, t]}$ is the indicator function of the interval $[0, t]$ (also called the Heaviside function).

Other cases are known in closed form as well. For example the case of the operator:
\begin{equation}\label{ito-gen}
L=Tr(A D^2 u) + b \cdot Du,
\end{equation}
in which case the Green's function for the corresponding operator is the probability density for the It\=o process:
\begin{eqnarray}\label{ito-proc}
dX_t = b dt + S dW_{t},
\end{eqnarray}
where $S$ is such that $A=S S^{T}$ and $W_t=(W_{1,t}, \cdots, W_{d,t})$ is the standard ${\mathbb R}^{d}$ Wiener process.

\subsection{Mild solutions}
Using the Green function, we may also express the solution of non-linear parabolic equations in terms of an equivalent integral equation, resulting in what is called the mild solution of the non-linear PDE. For the sake of concreteness, consider the nonlinear parabolic PDE:
\begin{equation}\label{44444}
\begin{aligned}
u_t &= L u(t,x) + F(t,x,u(t,x)), \, \,
u(0,x)&=u_{0}(x),
\end{aligned}
\end{equation}
where  $F : [0, T] \times {\cal D} \times {\mathbb R} \to {\mathbb R}$, is a known non-linear function, assumed for simplicity to be globally Lipschitz.  The problem is also subject to boundary conditions on $\partial {\cal D}$, or we may assume the problem on the whole space ${\mathbb R}^{d}$, subject to suitable decay conditions at infinity.

Using the Green's function $G$ for the operator ${\cal L}$, as well as \eqref{33333}, we can express \eqref{44444} in integral form in terms of (see \cite{boulle2022learning}):
\begin{equation}\label{55555}
\begin{aligned}
u(t,x)=\int_{{\cal D}} G(t, x ; 0, z) u_{0}(z) dz + \int_{0}^{T} \int_{{\cal D}} G(t,x ; s, z) F(s, z, u(s,z) ) dz ds =: ({\cal G} \star F(\cdot, u(\cdot) ) )(t,x).
\end{aligned}
\end{equation}

\begin{definition}[Mild solution]
A function $u \in C([0,T], X )$, where $X$ is a suitable Banach space, that satisfies the integral equation \eqref{55555} is called a mild solution of \eqref{44444}.
\end{definition}

A suitable choice for $X$ is a Sobolev space for example $X=W^{1,2}({\cal D})$.
Using fixed point arguments, the solvability of \eqref{55555} can be demonstrated (under suitable conditions), thus leading to the existence of mild solutions to \eqref{44444}. It should be noted that the integral representation \eqref{55555} allows for weaker solutions than the classical formulation of \eqref{44444}, existing under weaker assumptions on the data of the problem, e.g., $F \circ u$ being a $L^{1}$ function. In particular, \eqref{55555} allows for $u$ to be a continuous function of time, taking values in an appropriate function space carrying the spatial dependence of the solution. This is to be compared with the classical formulation in \eqref{44444} which requires $u$ to be continuously differentiable in $(t,x)$ for a sufficient number of derivatives such that \eqref{44444} makes sense everywhere in $[0, T] \times {\cal D}$. Clearly, subject to additional regularity assumptions on the data of the problem, the mild solution will be regular enough to qualify as a classical solution. In the general case, the notion of mild solutions allows for the existence of solutions for equations that do not admit classical solutions, and thus generalizes and extends the notion of a solution for both linear and non-linear PDEs.

\section{HEATNETs: Shallow random feature/ projection neural networks with heat-kernels}

%\subsection{An abstract result}

Consider the function space $S_{F}$ of solutions of the (non-linear) parabolic PDE \eqref{44444} 
\begin{equation}
S_{F}:=\{ u \,\, \mbox{satisfying} \,\, \eqref{44444} \,\, \mbox{with RHS} \,\, F\},
\end{equation}
for a particular choice of RHS $F$, and the solution space,
\begin{equation}
S=\bigcup_{F \in Z} S_{F}=\{ u \in X_{T} \,\, \mbox{satisfies} \,\, \eqref{44444} \,\, \mbox{for some} \,\, F \in Z\},
\end{equation}
where $X, Z$ are appropriate function spaces (a typical choice would be $Z=L^{p}(\Omega)$ for some $p \ge 1$) and $X_{T}$ some Sobolev space of functions on $[0, T] \times {\cal D}$  or continuous functions with values on $X$, a Sobolev space). 

To simplify the arguments used here, we make  (without loss of generality) the assumption that the functions in ${\cal S}$ admit continuous representations.

We will also define the set of random features neural networks
\begin{eqnarray*}
{\cal R}(w, \phi)= \{ u : [0, T] \times {\mathbb R}^{d} \to {\mathbb R}, \,\, : \,\, u(t,x)=\sum_{i=1}^{N} w_{i} \phi_{i}(t, x), \,\, N \in {\mathbb N} \}
\end{eqnarray*}
where $\phi=( \phi_1, \cdots, \phi_N)$ is a set of random feature functions, $\phi : [0, T] \times {\cal D} \to {\mathbb R}$
and $w=(w_1, \cdots, w_{N}) \in {\mathbb R}^{N}$ is a set of appropriate weights.  We are deliberately vague concerning the choice of random feature functions, but note that an important special case is the case where $\phi_{i}(\cdot, \cdot)=\Phi( \cdot, \cdot ; \theta_{i})$ where $\Phi : [0, T] \times {\cal D} \times \Theta \to {\mathbb R}$, is a central  parametric family of (deterministic) functions, depending on a parameter $\theta \in \Theta$, where $\Theta$ is an appropriate parameter set, and $\theta_{i} \sim_{iid} \Psi$, where $\Psi$ is an appropriate probability distribution.

We can formulate the following result.

\begin{proposition}\label{prop-representation-1}
There exists a shallow random feature neural network (RFNN), ${\cal R}(w, \phi)$, that can approximate arbitrarily close any function $u \in S$.
\end{proposition}

\begin{proof} For simplicity, without loss of generality, set the initial condition $u_0=0$.
Consider any function $u \in S$. By definition this is the solution of a (non-linear) parabolic PDE such as \eqref{44444}. We now switch to its mild formulation \eqref{55555}, and express it as
\begin{equation}\label{55555-1}
u(t,x)=\int_{0}^{T} \int_{{\cal D}}G(t, x ; s, z)  F(s, z , u(s, z)) dz ds.
\end{equation}
Let $Z = ( s, z) \sim {\cal U}([0, T] \times {\cal D})$ be a random variable uniformly distributed.

We note that one can express  \eqref{55555-1} in terms of the probabilistic  representation as
\begin{equation}\label{55555-2}
u(t , x) = {\mathbb E}_{Z}[G(t, x ; s, z)  F(s, z , u(s, z))].
\end{equation}
 Based on the above, we may choose a sample of points $\{\theta_i:=(s_i, z_i), \,\, i=1, \cdots, N\} \sim {\cal U}([0, T] \times {\cal D})$ and express the approximate the expectation in \eqref{55555-2} in terms of the estimator
\begin{equation}\label{55555-2}
\begin{aligned}
u(t , x) &= {\mathbb E}_{Z}[G(t, x ; s, z) F(s, z , u(s, z))]
\simeq  u_{N} :=\frac{1}{N} \sum_{i=1}^{N} G(t, x ; s_i, z_i)  F(s_i, z_i , u(s_i, z_i)).
\end{aligned}
\end{equation}
This approximation is essentially the Monte-Carlo approximation to the expectation in \eqref{55555-1}  and by the law of large numbers it is known that 
the estimator $u_{N}$ converges almost surely to $E:={\mathbb E}_{Z}[G(t, x ; s, z)  F(s, z , u(s, z))]$, with a rate of convergence as $O(N^{-1/2})$. We emphasize that the estimator $u_{N}$ is still in implicit form (as it depends on the function $u$, but this is not an issue for the arguments here.

Assuming the existence of a mild solution $u$, which allows for continuous pointwise estimation,  the function $F(\cdot, \cdot, u(\cdot, \cdot))$ is well defined hence upon defining the functions
\begin{eqnarray}\label{77777-0}
(t, x) \mapsto \phi(t, x ; \theta) := G(t , x ; s, z), \,\,\,\, \theta=(s, z),
\end{eqnarray}
as well as the weights
\begin{eqnarray}\label{pagrati-100}
w_{i}:=\frac{1}{N}  F(s_i, z_i, u(s_i, z_i)),
\end{eqnarray}
combining \eqref{55555-1}, \eqref{55555-2},  \eqref{77777-0} and \eqref{pagrati-100}, we end up with an approximation of the mild solution as
\begin{equation}\label{77777}
u_{N}(t, x) = w_i \phi(t, x ; \theta_i), \,\,\, \theta_i \sim {\cal U}([0, T] \times \mathcal{D}).
\end{equation}
We emphasize that the weights exist (are well defined) by the above arguments and importantly are independent of $(t,x)$.
This completes the proof.
\end{proof}

\begin{remark}\label{remark-assoe-1}
The representation \eqref{77777} can be interpreted as a shallow RFNN that can be used for the representation of the mild solution. The following remarks are in order.
%\textcolor{red}{With Red color below to be deleted}
\begin{itemize}

\item[(1)] The choice of randomization scheme presented in Proposition \ref{prop-representation-1}, i.e., in terms of $\theta_{i} \sim_{i.i.d} {\cal U}([0, T] \times {\cal D})$ is by no means exclusive. Alternative randomization schemes, akin to importance sampling schemes, which may enhance the performance of the approximation can be proposed (see Section \ref{sec-randomization}).

\item[(2)] Representation \eqref{77777} guarantees the existence of coefficients $w=(w_1, \cdots, w_N)$ such that the mild solution can be expressed as the linear combination of the random features $\phi(\cdot, \cdot ; \theta_i)$. Clearly, as seen in the proof, the coefficients $w$ depend on the solution $u$ we are aiming to find whereas the random features do not. This is of course natural and consistent with the spirit of other universal approximation theorems. Moreover, the coefficients $w$ are random variables, depending on the choice of the randomized parameters $\{\theta_i, \,\, i=1, \cdots, N\}$. 
\color{black}
\item[(3)] For general nonlinear problems or in even in the case where the exact Green's function is not known (see Section \ref{section-basis} below) the important question of determining the weights $w=(w_1, \cdots, w_N)$ in the expansion arises. This can be addressed using a PINN methodology, assuming an approximation of the solution as $u_{N}(\cdot, \cdot)=\sum_{i=1}^{N} w_{i} \phi_{i}(\cdot, \cdot)$, which satisfies the equation \eqref{44444} or its mild formulation \eqref{55555}.

\end{itemize}
\end{remark}

\subsection{Choosing the randomization scheme}\label{sec-randomization}

The basic idea behind the choice of the randomization scheme, is that we can choose a distribution $\Psi$, from which we can sample the values of the parameters of the randomized neural network $\mathcal{R}(w,\phi)$, so that the singularities of the Green's function can be smoothed out. At this point, we demonstrate the following proposition.

\begin{proposition}[Alternative sampling schemes]\label{prop-sampling}
    An alternative choice for the randomized neural network ${\cal R}(w, \phi)$ approximating arbitrarily close any function in $S$ can be in terms of the scheme
    \begin{equation}
\begin{aligned}
u(t,x)=\sum_{i=1}^{N} w_i \bar{\phi}(t,x ; \theta_i), \,\,\,\, \theta_i \sim \Psi,
\end{aligned}
\end{equation}
\color{black}
where
\begin{equation}
\begin{aligned}
(t, x) \mapsto \bar{\phi}(t, x ; \theta_i) = \frac{1}{\psi(s_i, z_i)} G(t, x ; s_i, z_i),
\end{aligned}
\end{equation}
where $\psi$ is the probability density function of the probability distribution $\Psi$.
\end{proposition}

\begin{proof} For simplicity, without loss of generality, set the initial condition $u_0=0$.
Starting from the representation of the mild solution as
\begin{equation}\label{999999}
\begin{aligned}
u(t,x)&=\int_{0}^{T} \int_{{\cal D}} G(t, x ; s, z) F(s, z, u(s, z)) dz ds \\
&=\int_{0}^{T}\int_{{\cal D}} 
G(t, x ; s, z) \frac{1}{\psi(s, z)} F(s, z, u(s, z))  \psi(s, z) dz ds \\
&={\mathbb E}_{\theta =(s, z) \sim \Psi}\left[ G(t, x ; s, z)  \frac{1}{\psi(s, z)} F(s, z, u(s, z))  \right] \\
& \simeq \frac{1}{N} \sum_{i=1}^{N} G(t, x ; s_i, z_i) \frac{1}{\psi(s_i, z_i)} F(s_i, z_i, u(s_i, z_i)), \,\,\,\, \theta_i =(s_i, z_i) \sim \Psi, 
\end{aligned}
\end{equation}
where $\psi$ is the probability density function for $\Psi$.

We may interpret representation \eqref{999999} as
a RFNN of the form
\begin{equation}
\begin{aligned}
u(t,x)=\sum_{i=1}^{N} w_i \bar{\phi}(t,x ; \theta_i), \,\,\,\, \theta_i \sim \Psi,
\end{aligned}
\end{equation}
where
\begin{equation}
\begin{aligned}
(t, x) \mapsto \bar{\phi}(t, x ; \theta_i) = \frac{1}{\psi(s_i, z_i)} G(t, x ; s_i, z_i),
\end{aligned}
\end{equation}
for an appropriate choice of
weights $w=(w_1, \cdots, w_N)$. The fact that such a choice of weights exists can be deduced from \eqref{999999} by setting
\begin{equation*}
w_i = \frac{1}{N} F(s_i, y_i, u(s_i, z_i)), \,\,\, i=1, \cdots, N.
\end{equation*}
\color{black}
This completes the proof.
\end{proof}
Clearly, the above Proposition, by the properties of mild solutions guarantees that such coefficients $w_i$ exist, which is sufficient for a universal approximation type of theorem. For any practical purposes, the calculation of these coefficients can be done for example within the PINN framework.

\subsection{Choosing the basis functions}\label{section-basis}

In this section, we show that for the representation of mild solutions of a large class of parabolic PDEs, the basis functions can always be chosen to be Gaussian RBFs, irrespectively of whether the corresponding parabolic operator ${\cal L}$ admits a Green's function which is a Gaussian function. Although this can be deduced by the universal approximation properties of Gaussian RBFs, our approach offers an alternative proof of this result, which is elementary and not based on abstract analytical results and at the same time, using simply the law of large numbers, establishes the approximation theorem in terms of a randomized shallow neural network. Furthermore, Gaussian/RBFs require (nonlinear) tuning shape parameters which scales poorly in high dimensions. Instead, randomized Green's functions they give naturaly a Monte-Carlo approximation to the convolution of the kernel, thus scaling better with the number of collocation points and dimension.

\begin{proposition}\label{balcony-pagkrati-1}
Let $L$ be a general elliptic operator and $L_0=Tr(A D^2 \cdot )$, where $A$ is a constant coefficient positive definite matrix.
Then any solution $u$ of the (non-linear) parabolic equation  \eqref{44444}, such that $(L-L_0)u$ admits continuous representatives,
%\begin{eqnarray}
%\frac{\partial u}{\partial t} = %L u + F(t, x, u),
%\end{eqnarray}
can be represented  in terms of a random feature neural network ${\cal R}(w, \phi)$, where $\phi$ is a Gaussian kernel function.
\end{proposition}

\begin{proof} For simplicity, without loss of generality, set the initial condition $u_0=0$. 

Let $u$ be a solution of \eqref{44444} such that the integral formulation  \eqref{55555} (importantly the integral on the RHS is well defined).  We express the equation in terms of 
\begin{equation}\label{balcony-pagkrati-2}
\begin{aligned}
\frac{\partial u}{\partial t} - L_{0} u = (L - L_0) u + F(t, x, u).
\end{aligned}
\end{equation}
By the assumption that $u$ is a solution of \eqref{44444} such that the integral representation \eqref{55555} the reformulation in \eqref{balcony-pagkrati-2} is well defined. This reformulation is the fundamental in our argument.

Let $G_0$ be the Green's function for the parabolic operator ${\cal L}_0 =\frac{\partial }{\partial t}- L_0$, which is a Gaussian kernel. We now express the equation in mild form as:
\begin{equation}
\begin{aligned}\label{assoee-1}
u(t, x) =\int_{0}^{T}\int_{{\cal D}} G_0(t, x ; s, z) \bigg[  (L- L_0)u(s, z) + F(s, z, u(s, z) \bigg] dz ds.
\end{aligned}
\end{equation}
The result follows from this representation working as in the proof of Proposition \ref{prop-representation-1}.  In particular, since by assumption the solution $u$ exists, we set it to a known function $u=w$ and express the RHS of  \eqref{assoee-1}, as
\begin{eqnarray}\label{assoee-2}
    J(t, x):=\int_{0}^{T}\int_{{\cal D}} G_0(t, x ; s, z) \bigg[  (L- L_0)w(s, z) + F(s, z, w(s, z) \bigg] dz ds.
\end{eqnarray}
Since $w$ is assumed as known, under the extra assumption that the function or a version of it  and the function $(L- L_{0}) w$, admit pointwise representation (this assumption can be checked using regularity results for the solution of problem \eqref{44444}),  we can approximate $I$ in terms of a Monte-Carlo approximation as
\begin{equation}\label{assoee-4}
J_{N}(t, x)= \sum_{i=1}^{N}
\underbrace{\frac{1}{N}\{(L- L_0) w(s_i, z_i) + F(s_i, z_i , w(s_i, z_i)) \}}_{:= w_i} G_{0}(t, x ; s_i, z_i),
\end{equation}
where $\theta_i :=(s_i, z_i)$, $i=1, \cdots, N$, is a sample of i.i. d. random variables, $\theta_i \sim_{i.i.d} {\cal U}([0, T] \times {\cal D})$.
Combining \eqref{assoee-1} and \eqref{assoee-4} we obtain the approximation
\begin{equation}\label{assoee-10}
u(t,x) \simeq J_{N}(t, x).
\end{equation}

We emphasize that the real valued random variables $w_i$ defined  in \eqref{assoee-4} above, do not depend on $(t, x)$, so that \eqref{assoee-4} can be  interpreted as a shallow random neural network ${\cal R}(w,\phi)$
with $w=(w_1, \cdots, w_{N})$ as in \eqref{assoee-4} above, and $\phi=(\phi_1, \cdots, \phi_N)$ the random features $\phi_{i}=G_0( \cdot, \cdot ; s_{i}, z_{i})$ with $\theta_{i}=(s_{i}, z_{i})$ as above, which is a set of random Gaussian features. We emphasize, that in the general case here, the random weights $w$ depend on the solution $u=w$. However, the point of this proposition is not to show an explicit representation, but rather the {\bf existence} of a shallow random neural network  with Gaussian features representing the solution. Since the solution is assumed to exist \eqref{assoee-4} combined with \eqref{assoee-1} allows us to conclude.
 
The argument certainly holds for classical solutions.

\end{proof}

 Proposition \ref{balcony-pagkrati-1} is important since it allows for a wide range of choices of representing functions
related to the Green's function of the operator ${\cal L}_{0} = \frac{\partial }{\partial t}- L_0$. Importantly for the choice 
$L_0=Tr(A D^2 \cdot )$  this provides us with the set of Gaussians. In the particular case that $A=D I_{d}$, with $D>0$ a constant,  $L_0=D \Delta$, the Laplacian,  and the  relevant Gaussian is of the form  \eqref{30-10-2025}.

\subsection{Implementation}
In this section we will consider various forms of the representation \eqref{33333}, that can be handled numerically. In particular, we will first see how we can mathematically eliminate the singularities that arise in the Gaussian kernels, resulting in a new form of the representation that can be approximated even by naive schemes (e.g., integral discretization and/ or naive Monte Carlo sampling). As a next step, we consider importance sampling, (based on Proposition \ref{prop-sampling}), as a different lens through which, we can look at the integral approximations and which allows us to consider high-dimensional problems, as well. This analysis is useful not only for direct Monte Carlo estimations but will inspire the custom RFNN models we construct, based on the theory above. 

We note that throughout the remainder of this paper we will be considering the linear PDE case, in order to introduce and benchmark the method. The theoretical framework however, as detailed above, provides the necessary results to extend the proposed methods to non-linear PDEs in a future work.

\subsubsection{Dealing with the singularities}\label{trasformations-section}

We first consider a change of variables that is required to ensure numerical accuracy of the representation \eqref{33333}.  As mentioned, one can see that the Gaussian kernels in the second term of \eqref{33333} exhibit singularities as $s \rightarrow t$ and $z \rightarrow x$. Hence, during the Monte Carlo sampling, problematic terms may occur by selecting $(s_i, z_j)$ such that $\|(s_i, z_j) - (t,x)\| < \delta$, for small enough $\delta$. On the other hand, excluding such values may lead to the loss of important data near the point of interest $(t,x)$, leading to large errors in the approximation, particularly in cases with steep gradients.
\par Hence, we are tasked with finding a way to deal with these singularities numerically. As can be done in many cases with integrable functions that blow-up, we will consider suitable transformations.
\par Depending on the spatial dimension $d$, we require slightly different handling. We outline each of the cases below, starting with the second integral in our representation \eqref{33333}. For convenience, we will adopt the following notation throughout the remainder of this paper:
\begin{equation}\label{I-integral}
   I(t,x)
=\int_{\mathbb{R}^d }\frac{1}{({4\pi D\,t})^{d/2}}
\exp\!\Bigl(-\frac{\|x - z\|^{2}}{4D\,t}\Bigr)\,u_{0}(z)\,\mathrm{d}z, 
\end{equation}

\begin{equation}\label{J-integral}
J(t,x) := \int_{0}^{t}\!\!\int_{\mathbb{R}^d}
  \frac{1}{\bigl(4\pi D\,(t-s)\bigr)^{\tfrac d2}}
  \exp\!\Bigl(-\frac{\|x-z\|^2}{4D\,(t-s)}\Bigr)
  \,F(s,z)\,dz\,ds.
\end{equation}

\paragraph*{Case \(d=1\)}
Here, the prefactor is \((t-s)^{-1/2}\) and by setting $\tau = \sqrt{\,t - s\,},$ we obtain:
\begin{equation}\label{transormed-J-1}
J(t,x)
=\int_{0}^{\sqrt t}
  \!\int_{\mathbb{R}} \frac{2}{\sqrt{4\pi D}}
    \exp\!\Bigl(-\frac{\|x-z\|^2}{4D\,\tau^2}\Bigr)
    \,F\bigl(t - \tau^2,\;z\bigr)\,dz
  \,d\tau.
\end{equation}

\paragraph*{Case \(d=2\)}
The prefactor is \((t-s)^{-1}\) and we make the change of variables $ \tau = -\ln\bigl(t-s\bigr)$ and get:
\begin{equation}\label{transormed-J-2}
J(t,x)
=\int_{-\ln t}^{\infty}
  \!\int_{\mathbb{R}^2} \frac{1}{4\pi D}
    \exp\!\Bigl(-\frac{\|x-z\|^2}{4D\,e^{-\tau}}\Bigr)
    \,F\bigl(t - e^{-\tau},\,z\bigr)\,dz
  \,d\tau.
\end{equation}

\paragraph*{Case \(d\geq 3\)}
Finally, in this general case we have the term $(t-s)^{-d/2}$ that can cause numerical blow-ups. In this case, we define:
\begin{equation}\label{transormed-J-ge2}
\alpha = 1 - \frac{d}{2},
\quad
\tau = (t-s)^{\,\alpha}.
\end{equation}
Then, we have:
\begin{equation}
J(t,x)
=\int_{t^\alpha}^{\infty}
 \!\int_{\mathbb{R}^d} \frac{1}{\lvert\alpha\rvert\,(4\pi D)^{d/2}}
   \exp\!\Bigl(-\frac{\|x-z\|^2}{4D\,\tau^{1/\alpha}}\Bigr)
   \,F\bigl(t - \tau^{1/\alpha},\,z\bigr)\,dz
 \,d\tau.
\end{equation}

We now move on to the homogeneous part of the Gaussian integral representation \eqref{I-integral}.
For this term, notice that as $t \rightarrow 0$ we again face a numerical blow-up. To approximate this term numerically avoiding large errors due to the singularity, we set $z = x + y\,\sqrt{4D\,t},$ with $\mathrm{d}z = \sqrt{4D\,t}\,\mathrm{d}y$. Then we get:
\begin{equation}\label{transformed-I}
  I(t,x)
=\int_{\mathbb{R}^d}\frac{1}{\sqrt{\pi}}\,e^{-\|y\| ^{2}}
\,u_{0}\bigl(x + y\,\sqrt{4D\,t}\bigr)\,\mathrm{d}y.  
\end{equation}
\par For the transformations, since we have differing versions according to the spatial dimension $d$, we will use $\mathcal{D}(t)$ for the transformed temporal domain over which we are integrating. For example, for $d=1, d=2$, $\mathcal{D}(t) := [0,\sqrt{t}]$ and $[-\ln t, \infty)$, respectively. Furthermore, we define the function $g:\mathcal{D}(t) \rightarrow \mathbb{R}_+$, to represent the function of $\tau$ obtained by the change of variables for the forcing term (e.g., for $d=1$, $g(\tau) = \tau^2$). Finally, for consistency and to easily distinguish between the two cases we will use $\mathcal{J}(t,x)$ and $\mathcal{I}(t,x)$ when referring to the transformed versions of the integrals, given by \eqref{transormed-J-1}-\eqref{transormed-J-ge2} and \eqref{transformed-I}, respectively. Therefore, we can write:
\begin{equation}\label{gen-transformed-J}
\begin{aligned}
    \mathcal{J}(t,x) =\int_{\mathcal{D}(t)}
 \!\int_{\mathbb{R}^d} \frac{1}{C(d)(4\pi D)^{d/2}}
   \exp\!\Bigl(-\frac{\|x-z\|^2}{4D\,g(\tau)}\Bigr)
   \,f\bigl(t - g(\tau),\,z\bigr)\,dz
 \,d\tau, \\
 \mathcal{I}(t,x) =\int_{\mathbb{R}^d}\frac{1}{\sqrt{\pi}}\,e^{-\|y\|^{2}}
\,u_{0}\bigl(x + y\,\sqrt{4D\,t}\bigr)\,\mathrm{d}y,
 \end{aligned}
\end{equation}
for any $d\geq 1$, and where $C(d)$ is a constant given by:
\begin{equation}
    C(d) = 
    \begin{cases}
        \frac{1}{2}, \,\,\,\ d = 1, \\
        1 \,\,\,\ d = 2, \\
        \lvert \alpha \rvert \,\,\,\ d \geq 3.
    \end{cases}
\end{equation}

We can now consider the naive Monte Carlo approximation of the integral representation efficiently. The transformed versions of the integrals provide expressions for which we can consider naive sampling schemes which do not exhibit singularities. To this end, we sample from $[-A,A]^{d}$, for some large enough constant $A>0$, drawing $\{y_{j}\}_{j=1}^{M}\sim{Unif}([-A,A]^d)$ with density $q(y)=1/{(2A)}^d$. This gives the unbiased Monte Carlo estimator:
\begin{equation}\label{transformed-I-MC}
\mathcal{I}(t,x)\approx
\frac{1}{M_0}
\sum_{j=1}^{M_0}
\frac{(2A)^d}{\sqrt{\pi}}\,e^{-\|y_{j}\|^{2}}\,
u_{0}\!\bigl(x+y_{j}\sqrt{4D\,t}\bigr).
\end{equation}
\par For $\mathcal{J}(t,x)$ we need to sample from the spatial and temporal domains, $\mathbb{R}^d$ and $\mathcal{D}(t)$. (Note that, if necessary such as for $d=3$, we truncate $\mathcal{D}(t)$). We then draw $\{\tau_j\}_{j=1}^{M_1} \sim {Unif}({\mathcal{D}}(t))$, $\{z_{j}\}_{j=1}^{M_1}\sim{Unif}([-A,A]^d)$ to approximate:
\begin{equation}\label{transformed-J-MC}
 \mathcal{J}(t,x) \approx
\frac{\mu({\mathcal{D}}(t))(2A)^{d}}
     {M\,C(d)(4\pi D)^{d/2}}
\sum_{j=1}^{M}
\exp\!\Biggl(
  -\frac{\lVert x - z_{j}\rVert^{2}}
        {4D\,g(\tau_{j})}
\Biggr)
F\!\bigl(t - g(\tau_{j}),\,z_{j}\bigr),
\end{equation}
where $\mu(A)$ represents the Lebesgue measure of the truncated interval $A$.

\par As an example, consider the case where $d \geq 3$. We fix a large enough $T > t^{\alpha}$ for the truncation of the temporal integral and drawn uniformly from $[t^{\alpha},T]$, $\{z_{j}\}_{j=1}^{M}\sim {Unif}([-A,A]^d)$, and the estimator is:
\begin{equation}\label{transformed-J-MC}
\mathcal{J}(t,x)
\approx
\frac{(T - t^{\alpha})\,(2A)^{d}}
     {M\,|\alpha|\,(4\pi D)^{d/2}}
\sum_{j=1}^{M}
\exp\!\Biggl(
  -\frac{\lVert x - z_{j}\rVert^{2}}
        {4D\,\tau_{j}^{\,1/\alpha}}
\Biggr) F\!\bigl(t - \tau_{j}^{\,1/\alpha},\,z_{j}\bigr).
\end{equation}

Despite the mathematical convenience of the transformed versions above, we note that these forms are mainly applicable for low dimensional problems. This is due to the effect of large exponents that still appear in \eqref{transformed-J-MC}, which for (very) large values of $d$ can lead to numerical intractability and high errors. This is dealt with in the following section.

\subsubsection{Importance sampling for high-dimensions}\label{sec-imp-samp}
In the above, we algebraically eliminated the singularities by means of an appropriate transformation of variables and then considered a naive sampling scheme for the MC approximation. Even though this approach indeed avoids any numerical blow-ups and provides high accuracy in low dimensions, high variance can occur for larger values of $d$. Indeed, from \eqref{transformed-I-MC} and \eqref{transformed-J-MC}, we can see the direct effects of the ``curse of dimensionality'', whereby extremely large values of the sample number $M$ will be required for accurate and low-variance estimations. This, coupled with the aforementioned numerical difficulties of the transformed kernels, motivate the following handling.
\par To circumvent this issue, we may consider the original forms of the integrals in combination with importance sampling. This approach will allow us to directly use the Gaussians as sampling distributions rather than functions to evaluate at distinct points, avoiding direct computation of regions near singularities. At the same time, the dimensionality of the problem is also ``absorbed'' by the sampling distribution, allowing us to avoid the blow-up due to high powers of the normalization constants, present in \eqref{transformed-J-MC}.  
\par Considering initially the integral $I(t,x)$ as given in \eqref{I-integral}, we have by definition: 
\begin{equation}
    I(t,x) = \mathbb{E}[u_0(Z)],
\end{equation}
where $Z\sim \mathcal{N}(x, 2DtI_d)$ with corresponding density function $G(t,x;0,z)$. Hence, we can generate $z_j = x+\sqrt{2Dt}\,\eta_j$ where $\{\eta_j\}_{j=1}^{M_0} \sim \mathcal{N}(0, I_d)$ and calculate: 
\begin{eqnarray}\label{I-integral-IS}
    {I}_{IS}(t,x) \approx \frac{1}{M_0} \sum_{j=1}^{M_0} u_0(z_j).
\end{eqnarray}
For the forcing integral $J(t,x)$ we will have to sample appropriately to account for the temporal and spatial dimensions: for a fixed $s$ value, the inner integral of $J(t,x)$, as given in \eqref{J-integral}, can be written as:
\begin{equation}
    \int_{\mathbb{R}^d} G(t,x;s,z)f(s,z)dz = \mathbb{E}[f(s,Z)|S=s],
\end{equation} 
where $Z|S=s \sim \mathcal{N}(x, 2D(t-s)I_d)$. Then, we have: 
\begin{equation}
    J(t,x) = \int_0^t \mathbb{E}[f(s,Z)|S=s]ds = \int_{0}^{\infty} {\bf 1}_{[0, t]}(s) \mathbb{E}[f(s,Z)|S=s]ds.
\end{equation}
Therefore, we can sample $S$ from the uniform distribution in $[0,t]$, by accounting for the density $q(s) = 1/t$ and 
% The resulting probability density is:
% \begin{equation}
%     q(s,z) = \frac{1}{t}G(t,x;s,z).
% \end{equation}
the resulting MC approximation is given by:
\begin{equation}\label{J-integral-IS}
    {J}_{IS}(t,x) \approx \frac{t}{{M_1}} \sum_{j=1}^{M_1} F(s_j, z'_j),
\end{equation}
where $s_j = r_jt$, $\{r_j\}_{j=1}^{M_1} \sim {Unif}([0,1]), \{\xi_j\}_{j=1}^{M_1} \sim \mathcal{N}(0, I_d)$ and $z'_j = x+ \sqrt{2D(t-s_j)}\,\xi_j$, for $j=1,\dots M_1$.
\par We note that extending this approach to forcing terms of the form $F(t,x,u)$, we would get the implicit form:
\begin{equation}\label{J-integral-IS}
    u(t,x) \approx  \frac{1}{M_0} \sum_{j=1}^{M_0} u_0(z_j) + \frac{t}{M_1} \sum_{j=1}^{M_1} F(r_jt, z'_j, u(r_jt, z'_j)),
\end{equation}
with the random samples as above.

 We now move on to show how these formulations of the integral representation will form the basis of the custom RFNNs.

\subsubsection{PINN implementation: architecture and training}\label{sec:HEATNETs}
Monte Carlo methods have well-known drawbacks, such as high computational requirements associated with high number of samples and the need to approximate the integrals at each point $(t,x)$. A different approach would be to take advantage of the power of neural networks as universal approximators, motivated by recent literature examining ``numerics-informed'' machine learning approaches, where classical numerical methods are used to inspire different machine learning architectures or training approaches.  In this work, we consider a Random Feature Neural Network inspired by the mild representation \eqref{33333} and the corresponding numerical implementations as given in section \ref{trasformations-section} and \ref{sec-imp-samp}, which will learn the appropriate weights so that the model is a universal approximator, rather than require an expensive Monte Carlo calculation at each $(t,x)$.

\par This can be done using PINNs. We first recall that in PINNs, the solution any given PDE is modeled as a neural network, which is trained using a loss function that encapsulates the ``physics'' of the PDE as well as the initial (and boundary, if applicable) conditions. The standard loss function construction then consists of two components: one for the PDE residual in the interior domain and one for the initial condition residual. Then, the approximation of the solution, $u_{\theta}(t,x)$, is calculated by minimizing the loss function, formulated as:
\begin{equation}\label{loss-pinn-pde}
\mathcal{L}_{\theta} = \frac{1}{N_{PDE}} \sum_{i=1}^{N_c} \Bigl( \mathcal{A}\left[u_{\theta}\right](t_i, x_i) \Bigr)^2+\frac{1}{N_b} \sum_{i=1}^{N_{IC}} \Bigl( u_{\theta}(0, x'_i) - u(0, x'_i) \Bigr)^2,
\end{equation}
where $\mathcal{A}$ is the operator $\mathcal{A}[u](t,x):= \frac{\partial u}{\partial t}(t,x) - D \Delta u(t,x) -  F(t,x, u)$ the right with $\{(t_i,x_i)\}_{i=1}^{N_{PDE}}$ being collocation points in $[0,T]\times \mathbb{R}^d$ and $\{x'_i\}_{i=1}^{N_{IC}}$ are collocation points in $\mathbb{R}^d$ over which the initial condition is calculated.
\par We now describe the proposed RFNN models, termed HEATNETs.

\paragraph{Model architecture and construction.}\label{section-RFNN}

\par As stated, we will be considering the linear PDE with $F(t,x,u) \equiv F(t,x)$. The HEATNET is constructed using either the Gaussian kernels or importance sampling terms as features. We describe both these cases explicitly.
\par Consider first the RFNN with transformed Gaussian kernels. The form is given by:
\begin{equation}\label{PINN-linear}
\tilde{u}(t,x)=
\sum_{j=1}^{M_0}w_j^{(0)}\,\varphi_j^{(0)}(t,x)+
\sum_{j=1}^{M_1}w_j^{(1)}\,\varphi_j^{(1)}(t,x),
\end{equation}
where 
\begin{equation}\label{features-transformed}
\begin{aligned}
\varphi_j^{(0)}(t,x)=
\frac{1}{\sqrt{\pi}}\,
e^{-\|y_j\|^2}\,
u_{0}\!\bigl(x + y_j\,\sqrt{4Dt}\bigr), \\
\varphi_j^{(1)}(t,x) =
\mathbf1_{\mathcal{D}}(t)(\tau_i)\,
\frac{1}{C(d){4\pi D}^{d/2}}\,
\exp\!\Bigl(-\frac{(x - z_j)^2}{4D\,g(\tau_j)}\Bigr)\,
F\!\bigl(t-g(\tau_j),z_j\bigr), 
\end{aligned}
\end{equation}
and random samples $\{y_j\}_{j=1}^{M_0}, \{z_j\}_{j=1}^{M_1}\sim Unif([-A,A]^d)$ and $\{\tau_j\}_{j=1}^{M_1} \sim  Unif\bigl({\mathcal{D}}(T)\bigr)$, for some fixed time horizon $T>0$. The random samples are generated during the construction of the neural network and remain fixed throughout the subsequent training and testing. Therefore, note that since $\mathcal{J}(t,x)$ depends on the input value $t$, our samples $\tau_j$ must sample the entire temporal domain during the construction of the neural network. Hence, $\varphi^{(1)}$ includes the indicator term that activates only the nodes corresponding to the $\tau$ samples within the acceptable range indicated by the input of the neural network $t$. 

\par Alternatively, we can construct the random features using the importance sampling representation. In this case we get \eqref{PINN-linear} with the random features given by:

\begin{equation}\label{features-is}
\begin{aligned}
\varphi_j^{(0)}(t,x)=
u_{0}\!\bigl(x + \sqrt{4Dt}\, \eta_j\bigr), \\
\varphi_j^{(1)}(t,x) = t F( r_i t, x+ \sqrt{2D(t-s_it)}\xi_i)
\end{aligned},
\end{equation}
where $s_i = r_it$ and we sample $\{\eta_j\}_{j=1}^{M_0}, \{(r_j, \xi_j)\}_{j=1}^{M_1}$ as described in section \ref{sec-imp-samp}, during the model construction. We will refer to these two version of the RFNN as $\tilde{u}_G(t,x)$ and $\tilde{u}_{IS}(t,x)$ for the Gaussian and importance sampling random features, respectively. This custom construction is shown schematically in Fig. \ref{fig:RFNN}.

\par As a final important note, by Proposition \ref{prop-sampling}, it is clear that many other HEATNETs can be constructed, using different sampling techniques. In addition, Quasi-Monte Carlo methods can be used to generate low-discrepancy points which can improve model variance (we will consider such an implementation using Sobol sampling in examples in the subsequent section). 

\begin{figure}
\centering
\begin{tikzpicture}[
  x=2.6cm, y=1.2cm,
  node distance=1cm and 2cm,
  every neuron/.style={circle, draw, minimum size=9mm, inner sep=0pt},
  input neuron/.style={circle, draw, minimum size=9mm, inner sep=0pt},
  output neuron/.style={circle, draw, minimum size=9mm, inner sep=0pt},
  >=stealth
]
% Inputs
\node[input neuron] (I1) at (0,  0.3) {$t$};
\node[input neuron] (I2) at (0, -0.9) {$x$};

% Hidden layer (6 neurons)
\node[every neuron, fill=green!25] (H1) at (1,  1.6) {$\varphi^{(0)}_1$};
\node[every neuron, fill=green!25] (H2) at (1,  0.8) {$\vdots$};
\node[every neuron, fill=green!25] (H3) at (1,  0.0) {$\varphi^{(0)}_{M_0}$};
\node[every neuron, fill=red!25]   (H4) at (1, -0.8) {$\varphi^{(1)}_1$};
\node[every neuron, fill=red!25]   (H5) at (1, -1.6) {$\vdots$};
\node[every neuron, fill=red!25]   (H6) at (1, -2.4) {$\varphi^{(1)}_{M_1}$};

% Output
\node[output neuron] (O) at (2, -0.2) {$\tilde{u}(t,x)$};

% Connections: inputs to hidden
\foreach \I in {I1,I2}
  \foreach \H in {H1,H2,H3,H4,H5,H6}
    \draw[->] (\I) -- (\H);

% % Connections: hidden to output
% \foreach \H in {H1,H2,H3,H4,H5,H6}
%   \draw[->] (\H) -- (O);
% --- Hidden -> Output with labels ---

% % Green edges (H1..H3): w_1^{(0)}..w_3^{(0)}
% \foreach \i in {1,2,3}{
%   \draw[->] (H\i) -- node[midway, sloped, above, font=\scriptsize, text=green!50!black]
%     {$w_{\i}^{(0)}$} (O);
% }

% Green edges with first, dots, last
\draw[->] (H1) -- node[midway, font=\scriptsize, text=green!50!black] {$w^{(0)}_1$} (O);
\draw[->] (H2) -- node[midway,  font=\scriptsize, text=green!50!black] {$\cdots$} (O);
\draw[->] (H3) -- node[midway, font=\scriptsize, text=green!50!black] {$w^{(0)}_{M_0}$} (O);

\draw[->] (H4) -- node[midway,  font=\scriptsize, text=red!50!black] {$w^{(1)}_1$} (O);
\draw[->] (H5) -- node[midway,  font=\scriptsize, text=red!50!black] {$\cdots$} (O);
\draw[->] (H6) -- node[midway, font=\scriptsize, text=red!50!black] {$w^{(1)}_{M_1}$} (O);

% --- Equations on the right -----------------------------------------------
% (tune the x–y positions and text width to your page)
\node[anchor=west, align=left, text=green!50!black,
      text width=7.0cm, draw, dashed, rounded corners,
      inner sep=4pt, fill=white, label={[font=\footnotesize]90:Initial condition feature}]
  (Eq0) at (2.2,0.9)
  {$\varphi^{(0)}_{j}(t,x)= 
  \begin{cases}
       \frac{1}{\sqrt{\pi}}\,
e^{-\|y_j\|^2} u_{0}\!\bigl(x + y_j\,\sqrt{4Dt}\bigr), \,\ \text{for Gaussian} \\
u_{0}\!\bigl(x + \sqrt{4Dt}\, \eta_j\bigr), \,\ \text{for Importance Sampling}

  \end{cases}
  $
  };

\node[anchor=west, align=left, text=red!60!black,
      text width=12cm, draw, dashed, rounded corners,
      inner sep=4pt, fill=white, label={[font=\footnotesize]90:Forcing term feature}]
  (Eq1) at (1.5,-2.2)
  {$\varphi^{(1)}_{j}(t,x) = 
  \begin{cases}
\mathbf1_{\mathcal{D}(t)}(\tau_i)\,
\frac{1}{C(d){4\pi D}^{d/2}}\,
\exp\!\Bigl(-\frac{(x - z_j)^2}{4D\,g(\tau_j)}\Bigr)\,
F\!\bigl(t-g(\tau_j),z_j\bigr), \,\ \text{for Gaussian} \\
t F( r_i t, x+ \sqrt{2D(t-r_i t)}\xi_i), \,\ \text{for Importance Sampling}
\end{cases}
  $};

% --- Dashed connectors -----------------------------------------------------
\draw[dashed,->,>=stealth, green!50!black]
  (H2.east) .. controls +(0.7,0.0) and +(-0.7,0.25) .. (Eq0.west);

\draw[dashed,->,>=stealth, red!60!black]
  (H5.east) .. controls +(0.7,0.0) and +(-0.7,0.0) .. (Eq1.west);

\end{tikzpicture}
\caption{The Random Feature Neural Network used in the HEATNETs framework: The model consists of a single hidden layer with $M = M_0+M_1$ nodes. Each node corresponds to a random feature with $\varphi^{(0)}_j(t,x), \varphi^{(1)}_j(t,x)$ given by \eqref{features-transformed} or \eqref{features-is}. The weights from the inputs $t \in [0,T], x\in \mathbb{R}^d$ are units, by construction, and the weights from the hidden layer to the output are $w = [w^{(0)} \,\,\ w^{(1)}]^T$. }\label{fig:RFNN}
\end{figure}
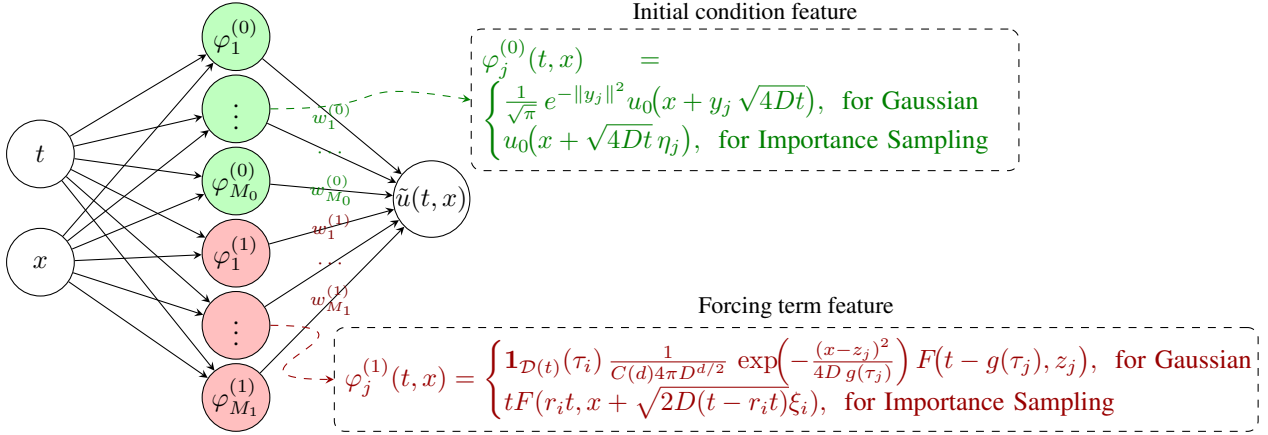

\paragraph{Model training.} \label{HEATNET-training}
\par As discussed, to train the model, we apply the  PINN approach, as given in \eqref{loss-pinn-pde}. Notice that the RFNN has a total of $M = M_0 + M_1$ nodes and is a linear function of the unknown weights vector $w = [w^{(0)}\,\, w^{(1)}]^T$, corresponding to the random feature vector $\varphi = [\varphi^{(0)}\,\, \varphi^{(1)}]^T$. Hence, differentiating the model \eqref{PINN-linear} is straightforward, as the derivatives pass to the features to obtain $\phi_t^{(i)}(t,x), \phi_x^{(i)}(t,x)$ and $\Delta \phi^{(i)}(t,x)$, for $i = 0,1$.

\par Therefore, training  becomes a simple Ordinary Least Squares (OLS) problem, as follows: we enforce the PDE at $N_{PDE}$ collocation points $\{(t_i,x_i)\}_{k=1}^{N_{PDE}} \subset (0,T] \times \mathbb{R}^d$, yielding the linear system
\begin{equation}
    A_{\rm PDE}\,w = b_{\rm PDE},
\end{equation}
where $A_{PDE} \in \mathbb{R}^{N_{PDE} \times M}, b_{PDE} \in \mathbb{R}^{N_{PDE}}$ are given by:
\begin{equation}\label{linear-ols}
   \bigl(A_{\rm PDE}\bigr)_{k,j}
=\Bigl[\partial_t\varphi_j - D\Delta\varphi_j(t_k,x_k) \Bigr],
\quad
(b_{\rm PDE})_k = F(t_k,x_k).
\end{equation}
Enforcing the initial condition at $N_{IC}$ points $\{x_{\ell}\}_{\ell=1}^{N_{IC}}$ gives 
\begin{equation}
A_{\rm IC}\,w = b_{\rm IC},
\end{equation}
with entries 
$(A_{\rm IC})_{\ell,j}=\varphi_j(0,x_\ell)$
and
$(b_{\rm IC})_\ell=u_0(x_{\ell})$,
$A_{IC} \in \mathbb{R}^{N_{IC} \times M}$, $b \in \mathbb{R}^{N_{IC}}$. Putting both blocks together yields the least‐squares problem:
\begin{equation}\label{eq:weighted_ridge_obj}
\hat{w} \;=\; \arg_{w}\min\;
\bigl\|A_{\mathrm{PDE}}w-b_{\mathrm{PDE}}\bigr\|_2^2
\;+\;\lambda_{{IC}}^{\,2}\,\bigl\|A_{\mathrm{IC}}w-b_{\mathrm{IC}}\bigr\|_2^2
\;+\;\lambda_{\mathrm{ridge}}\,\|w\|_2^2,
\end{equation}
where constants $\lambda_{IC}$ is the appropriate weighting of the initial condition requirement and $\lambda_{\mathrm{ridge}}$ is the regularization coefficient (if required). This can be written alternatively by constructing $A \in \mathbb{R}^{N\times M}, b\in \mathbb{R}^N$, with $N=N_{PDE} + N_{IC}$, as:
\begin{equation}
\label{eq:stacked_blocks}
A \;=\;
\begin{pmatrix}
A_{\mathrm{PDE}}\\[2pt]
\sqrt{\lambda_{\mathrm{IC}}}\,A_{\mathrm{IC}}
\end{pmatrix},
\qquad
b \;=\;
\begin{pmatrix}
b_{\mathrm{PDE}}\\[2pt]
\sqrt{\lambda_{\mathrm{IC}}}\,b_{\mathrm{IC}}
\end{pmatrix},
\end{equation}
and solving the normal equations:
\begin{equation}
\label{eq:ridge_normal_eqs}
\bigl(A^{\!\top}A+\lambda_{\mathrm{ridge}}I\bigr)\,w
\;=\; A^{\!\top}b.
\end{equation}
Naturally, when $\lambda_{ridge} = 0$, we obtain the solution as: 
\begin{equation} \label{least-square-pinn}
\hat{w}= A^{+}b,
\end{equation}
where $A^+$ is the Moore-Penrose inverse. The full algorithm is given in Algorithm \ref{algorithm-RFNN}.

Concluding, it is important to remark on the usefulness of the proposed RFNN implementation; by implementing the PINN training approach, we obtain a universal solution that still approximates the Monte Carlo sum pointwise (either $u_{G}(t,x)$ or $u_{IS}(t,x)$), but also provides the solution to the PDE for any $(t,x)$. In practice, this allows us to consider significantly fewer random samples in the RFNN (compared to those needed for the standard Monte Carlo estimations), and the corresponding learned weights are globally optimal for all choices $(t,x) \in [0,T] \times \mathbb{R}^d$. In other words, HEATNETs are a ``best-of-both worlds" approach, whereby the model directly simulates the closed-form mathematical expression, but the coefficients/parameters are determined efficiently and inexpensively using machine learning training.

\begin{algorithm}[htp!]
\caption{HEATNETs approach: RFNN construction (with Gaussian or importance sampling features) and training}\label{algorithm-RFNN}
\begin{algorithmic}[1]
\Require Time horizon $T>0$, truncated spatial domain $[-A,A]^d \subset\mathbb{R}^d $ and corresponding $[-\tilde{A}, \tilde{A}]^d \subset [-A,A]^d$, $\,u_0(x),\,F(t,x)$; sizes $M_0,M_1, M = M_0 + M_1$; counts $N_{\mathrm{PDE}},N_{\mathrm{IC}}$, $N = N_{PDE} + N_{IC}$; weights $\lambda_{\mathrm{IC}},\lambda_{\mathrm{ridge}}$
\Ensure  weights $w = [{w}^{(0)} \,\, w^{(1)}]^T \in\mathbb{R}^{M}$, where $w^{(0)} \in \mathbb{R}^{M_0}, w^{(1)} \in \mathbb{R}^{M_1}$ and model $\tilde{u}(t,x)$ ($\tilde{u}_{G}(t,x)$ or $\tilde{u}_{IS}(t,x)$, as defined in section \ref{section-RFNN}).
\State \textbf{Sample and freeze features:}
       draw $\{y_j\}_{j=1}^{M_0}, \{z_j\}_{j=1}^{M_1}\sim Unif([-A,A]^d)$ and $\{\tau_j\}_{j=1}^{M_1} \sim Unif\bigl({\mathcal{D}}(T)\bigr)$ (equivalently, $\{\eta_j\}_{j=1}^{M_0} \sim \mathcal{N}(0,I_d), \{r_j\}_{j=1}^{M_1}\sim Unif([0,1]), \{\xi_j\}_{j=1}^{M_1}\sim \mathcal{N}(0,I_d) $ for the importance sampling features).
(Alternatively for a Quasi-Monte Carlo approach, draw low-discrepancy points using e.g., Sobol sequences).
% draw low–discrepancy points using independent scrambled Sobol sequences:
% \(U^{(\eta)}\!\in[0,1]^{M_0\times d},\; U^{(\xi)}\!\in[0,1]^{M_1\times d},\; U^{(u)}\!\in[0,1]^{M_1}\).
% Map to the target feature distributions elementwise:
% \[
% \eta_j \;=\; \Phi^{-1}\!\big(U^{(\eta)}_{j}\big),\qquad
% \xi_i \;=\; \Phi^{-1}\!\big(U^{(\xi)}_{i}\big),\qquad
% u_i \;=\; U^{(u)}_{i},
% \]
% where \(\Phi^{-1}\) is the normal probit (inverse Gaussian CDF), yielding
% \(\eta_j,\xi_i \sim \mathcal{N}(0,I_d)\) and \(u_i\sim\mathrm{Unif}[0,1]\).
% (For the transformed–Gaussian feature variant, use
% \(y_j = -L + 2L\,U^{(\eta)}_{j}\) and \(z_i = -L + 2L\,U^{(\xi)}_{i}\)
% instead.) 

\State \textbf{Define feature maps:} $\varphi = [\varphi^{(0)} \,\, \varphi^{(1)}]^T \in \mathbb{R}^{M}$, with $\{\varphi^{(0)}_m\}_{m=1}^{M_0}, \{\varphi^{(1)}_m\}_{m=1}^{M_1}$ as given in \eqref{features-transformed} (equivalently \eqref{features-is} for the importance sampling RFNN).
\State \textbf{Training points:} sample $\{(t_k,x_k)\}_{k=1}^{N_{\mathrm{PDE}}}\subset (0,T]\!\times [-\tilde{A}, \tilde{A}]^d$ and $\{x_\ell\}_{\ell=1}^{N_{\mathrm{IC}}}\subset [-A,A]^d$.
\State \textbf{Build PDE and IC condition blocks:} \begin{itemize}
    \item $A_{PDE} \in \mathbb{R}^{N_{PDE} \times M}$, $b_{PDE} \in \mathbb{R}^{N_{PDE}}$ as $(A_{\mathrm{PDE}})_{km}\!=\!(\partial_t\varphi_m-D\Delta\varphi_m)(t_k,x_k)$ (via automatic or analytic differentiation) and $(b_{\mathrm{PDE}})_k\!=\!F(t_k,x_k)$.
    \item $A_{IC} \in \mathbb{R}^{N_{IC} \times M}$, $b_{IC} \in \mathbb{R}^{N_{IC}}$ as $(A_{\mathrm{IC}})_{\ell m}\!=\!\varphi_m(0,x_\ell)$ and $(b_{\mathrm{IC}})_\ell\!=\!u_0(x_\ell)$.
    \item Create $A \in \mathbb{R}^{N\times M}$, $b \in \mathbb{R}^N$ by stacking $A=\begin{pmatrix}A_{\mathrm{PDE}}\\ \sqrt{\lambda_{\mathrm{IC}}}\,A_{\mathrm{IC}}\end{pmatrix}$,\;
       $b=\begin{pmatrix}b_{\mathrm{PDE}}\\ \sqrt{\lambda_{\mathrm{IC}}}\,b_{\mathrm{IC}}\end{pmatrix}$.
\end{itemize}

\State \textbf{Solve ridge–LS:}
       obtain $\hat{w}$ as the solution to $(A_\lambda^\top A_\lambda+\lambda_{\mathrm{ridge}}I_{M \times M})w=A_\lambda^\top b_\lambda$.
\State \textbf{Return} $\tilde{u}(t,x)=\sum_{m=1}^{M_0} \hat{w}^{(0)}_m\,\varphi^{(0)}_m(t,x) + \sum_{m=1}^{M_1} \hat{w}^{(1)}_m\,\varphi^{(1)}_m(t,x) $.
\end{algorithmic}

\end{algorithm}

\section{Numerical examples}\label{num-examples}
 
To assess the performance of HEATNETs, we apply them for the solution of various parabolic PDEs with known analytical solutions, starting from 1D problems and extending up to high-dimensional cases reaching for our illustrations up to 2000 dimensions. All the experiments that follow were run using a Google Collab Python Notebook service with an A100 GPU with 80GB memory. 

\subsection{HEATNETs in 1D}

As a first example, consider the 1D parabolic PDE:
\begin{equation}\label{ex1}
u_t = D\,u_{xx} + (t+1)\sin(x),
\qquad
u(x,0)=\sin(x).
\end{equation} 

\par We let $D = 1$ and the true solution is $u(t,x) = (t+e^{-t})\sin(x)$. We consider a RFNN with the transformed Gaussian random features, using both random and Sobol low-discrepancy sampling, as well as the importance sampling features, for comparison. For all three tests, we use $M_0 = 32, M_1 = 64$, a time horizon $T = 1.0$ and a grid of collocation points in $(0,T] \times [-\pi,\pi]$, of size $N_{PDE} = 3,000$ and of size $N_{IC} = 1,000$ for the initial condition. For training the initial condition weighting is set as $\lambda_{IC} = \sqrt{3}$ enforcing equal weighting to the initial condition and PDE residual, and we set $\lambda_{\mathrm{ridge}} =0$, solving the minimization problem by the Moore-Penrose inverse. Fig. \ref{fig:linear-PINN} shows the results of the learned model on a new test grid of size $100\times 100$ in $[0,T]\times [-\pi/2, \pi/2]$, in order to eliminate the effects of the errors due to the free-boundary. (For brevity, we only plot the solution using the Gaussian RFNN, as the contours are identical.) 

\begin{figure}[h!]
  \centering
  \begin{subfigure}[t]{0.48\textwidth}
    \centering
    \includegraphics[width=0.75\linewidth]{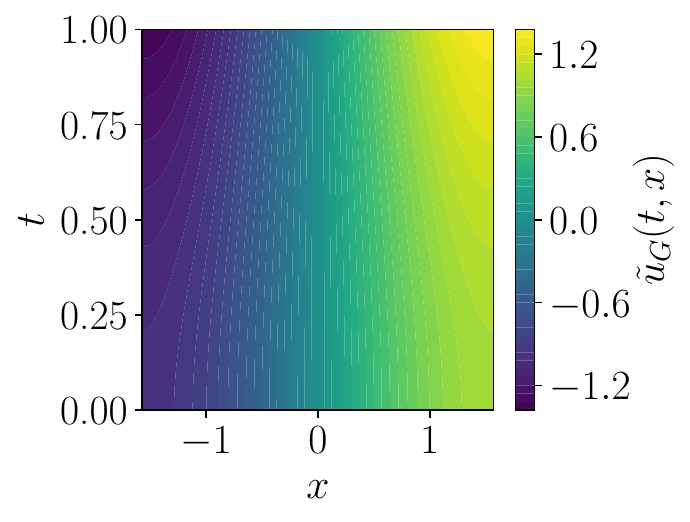}
    \caption{Approximation by the RFNN with transformed Gaussian features, $\tilde{u}_G(t,x)$.}
  \end{subfigure}
  \hfill
  \begin{subfigure}[t]{0.48\textwidth}
    \centering
    \includegraphics[width=0.75\linewidth]{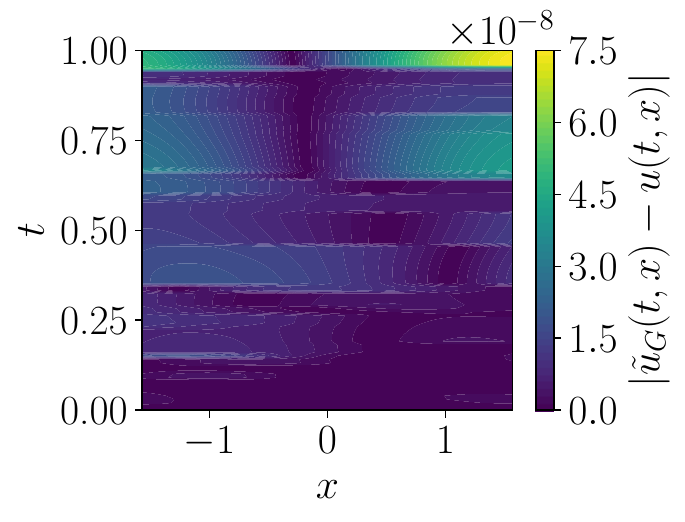}
    \caption{Absolute error of the RFNN with Gaussian features, $\tilde{u}_G(t,x).$}
  \end{subfigure} \\

  \begin{subfigure}[t]{0.48\textwidth}
    \centering
    \includegraphics[width=0.75\linewidth]{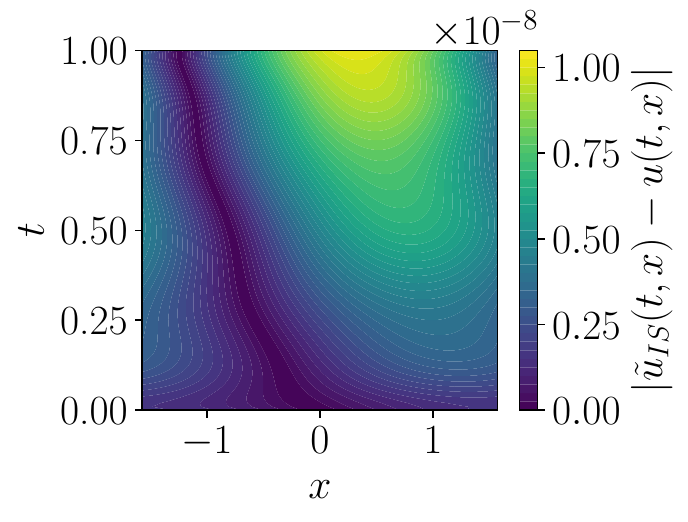}
    \caption{Absolute error of the RFNN with importance sampling features, $\tilde{u}_{{IS}}(t,x)$.}
  \end{subfigure}
  \hfill
  \begin{subfigure}[t]{0.48\textwidth}
    \centering
    \includegraphics[width=0.75\linewidth]{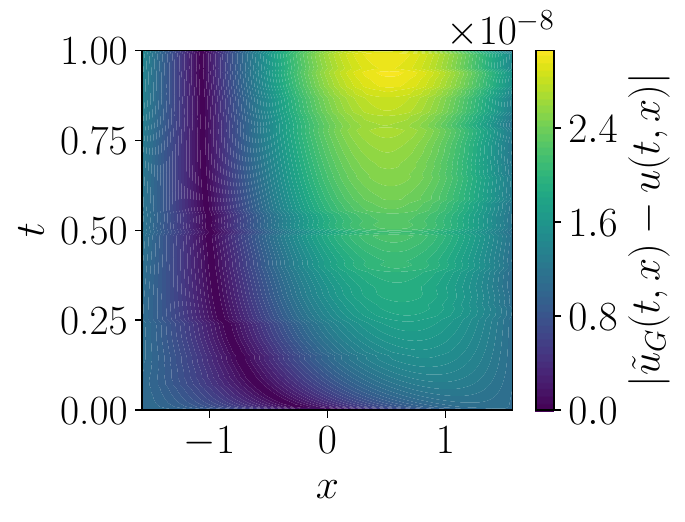}
    \caption{Absolute error of the RFNN with Gaussian features and Sobol sampling.}
  \end{subfigure}

  \caption{HEATNETs results for \eqref{ex1} on test grid. The RFNNs were constructed using $M_0=32$ and $M_1 =64$ features for the initial condition and the forcing term features, respectively and we consider 4,000 training points ($N_{PDE} = 3,000$ and $N_{IC} = 1,000$). }
  \label{fig:linear-PINN}
\end{figure}

\par Notice in the above example how the number of features in the HEATNET is many orders of magnitude less that the samples that would be required to calculate the Monte Carlo estimation of the solutions directly. This is the power of the PIML framework we mentioned in section \ref{sec:HEATNETs}: our custom neural network is trained so that the model learns the correct weights needed to make this representation a universal approximator.

\subsection{HEATNETs in high dimensions}

As seen in the theoretical derivations, standard difficulties with the naive Monte Carlo estimate in high dimensions motivate the use of importance sampling random features as the basis of the custom RFNN, which can help overcome the sampling difficulties. We proceed with this construction and the training of the HEATNETs as detailed in section \ref{section-RFNN}. 
\par As a starting point, we begin with the example below, up to dimension $d=10$, (due to its numerical properties, as will be explained). In the following examples, we provide a thorough analysis, using different sampling methods, training data and time horizon experiments going up to 2000 dimensions.

\subsubsection{Benchmark Example 1: PDE up to $d=10$}
We consider first the simple diffusion PDE (with a zero forcing term). This example is solved in \cite{von2004numerical}, up to dimension $d=5$ and for $T = 0.05$:
\begin{equation}\label{example-sch}
\begin{aligned}
    u_t(t,x) -D\Delta u(t,x) =0, \,\,\text{with}\\
    u(t,x) = \exp{\left(-d\pi^2 t \right)} \prod_{i=1}^d \sin(\pi x_i), \,\,
    u_0(x) = \prod_{i=1}^d \sin(\pi x_i).
\end{aligned}
\end{equation} 

\par The solution tends to zero exponentially in $d$.  For the construction of the HEATNET, $u_{IS}(t,x)$, notice that we only have the samples arising from the initial condition features. Hence, we take many samples corresponding to $\phi^{(0)}$, setting $M = M_0 = 15,000$. We train with $N_{PDE} = 20,000, N_{IC} = 4,000$ points and use $\lambda_{IC} = \sqrt{5}$ for the weighting of the initial condition and $\lambda_{\mathrm{ridge}} = 1.0$E$-06$ for the least squares problem. The relative $L_1, L_2$ and $L_{\infty}$ error metrics for the same time horizon and for dimensions up to $d=10$ are gathered in Table \ref{tab:sch_metrics_T005}.  (Note that these are the error metrics we consider throughout the remaining examples).

\begin{table}[ht]
\centering
\begingroup
\small
\setlength{\tabcolsep}{6pt}%
\renewcommand{\arraystretch}{1.05}%
\begin{tabular}{@{}lccc@{}}
\toprule
& \multicolumn{1}{c}{\(d=2\)} & \multicolumn{1}{c}{\(d=5\)} & \multicolumn{1}{c}{\(d=10\)} \\
\midrule
\(\mathrm{Rel.}\,L_1\)      & \(9.09\mathrm{E}{-08}\) & \(1.66\mathrm{E}{-07}\) & \(1.58\mathrm{E}{-05}\) \\
\(\mathrm{Rel.}\,L_2\)      & \(9.70\mathrm{E}{-08}\) & \(1.49\mathrm{E}{-07}\) & \(4.71\mathrm{E}{-06}\) \\
\(\mathrm{Rel.}\,L_{\infty}\) & \(1.33\mathrm{E}{-07}\) & \(1.61\mathrm{E}{-07}\) & \(1.25\mathrm{E}{-06}\) \\
Time (min)                  & \(1.52\) & \(1.97\) & \(2.77\) \\
\bottomrule
\end{tabular}
\endgroup
\caption{Error metrics and model building and training times for the Example \eqref{example-sch}, with $T =0.05$.}
\label{tab:sch_metrics_T005}
\end{table}
As shown, the HEATNET is able to match the errors of order E-07, as presented in \cite{von2004numerical}, and solve for higher dimensions $d =10$ while maintaining highly accurate predictions of the order E-05 - E-06 (in \cite{von2004numerical} the authors consider up to $d$=5). 
Finally, we also note that in \cite{gaby2024neural} the authors also consider a similar PDE and initial conditions, achieving relative errors of the order E-05 for dimensions up to $d=5$. 

\subsubsection{Benchmark Example 2: PDE up to $d = 2,000$}
We now consider the PDE for $x \in \mathbb{R}^d$:

\begin{equation}\label{example-high-d}
\begin{aligned}
    u_t(t,x) = D\Delta u(t,x) + g'(t)S_k(x) + Dg(t)k^2\,S_k(x), \,\,\ \text{with}\\
S_k(x)=\dfrac{1}{\sqrt{C}}\sum_{i=1}^d c_i \sin(k x_i), \,\,\,
g(t) = t+e^{-t}, \,\, \text{and } C=\sum_{i=1}^d c_i^2.
\end{aligned}
\end{equation}
The initial condition is $u(0,x) = S_k(x)$ and the true solution is given by $u(t,x) = g(t)S_k(x)$. For our illustrations, we will consider $k=2$ and $c_i = 1$, for all $i$. 
\par To demonstrate the full capacity and use-cases of the HEATNETs, we considered four sets of experiments. In the first, we use $M_0 = 500$ and $M_1 =1000$ samples corresponding to the homogeneous and the forcing integrals, respectively. We consider the domain $[0,T]\times [-\pi, \pi]^d$, $N_{PDE} = 10,000$ and $N_{IC} = 2,000$ training points for the PDE residual and initial condition respectively, with a $\lambda_{IC} = \sqrt{5}$ and a small ridge constant $\lambda_{\mathrm{ridge}} = 1.0$E$-06$. We run training instances for each of the combinations of dimensions and time horizons $(d,T)$ with various selections from $d=5$ to $d = 1,000$ and $T = \{0.25, 0.50, 0.75, 1.00\}$. The models are tested on a randomly sampled test grid $\{t_i,x_i\}_{i=1}^{N_{test}} \in [0,T]\times [-\pi/2, \pi/2]^d$, of size $N_{{test}} = 6,000$, by comparing the model output $\tilde{u}_{IS}(t_i,x_i)$ to the true solution $u(t_i,x_i)$. Secondly, to show that these results can be further improved when moving on to even higher dimensions by increasing the number of features, we plot the same graphs using $M_0=3,000$ and $M_1 = 5,000$. The error metrics for both these sets of experiments are gathered and shown in Fig. \ref{fig:high-dim-PINN-graphs} (note that the confidence bands across runs were omitted in these graphs as they were not visible). 

\par Finally, in Table \ref{tab:error_metrics_by_d_M}, we also provide the accuracy that can be achieved by increasing the number of features in the HEATNET and the corresponding training samples (we consider $N_{PDE} = 15,000$, $N_{IC} = 3,000$). This allows us to consider even higher dimensions, reaching $d = 2,000$. For clarity, we compare the results when using $M = 8,000$ and $M = 10,000$ features. As shown, increasing the features and training samples (as well as the corresponding computational requirements) results in highly accurate predictions across dimensions. 

\par As a final experiment, we also consider an under-determined problem, with $M_0 = 3,000, M_1 = 5,000$ and $N_{PDE} = 500, N_{IC} = 100$, for a single time horizon $T = 0.5$. To take advantage of the benefits of low-discrepancy sequences, we construct the HEATNET using Sobol sampling. The 10/90 percentile band, $[P_{10}, P_{90}]$, interquartile range and the median line, $P_{50}$, across 20 independent training instances are displayed in Fig. \ref{fig:high-dim-PINN-graphs-under}, showing that even with less training data the expressive power of the Gaussian features can produce accurate models, able to generalize and learn the underlying solution structure from sparse information. (We note that we do not include confidence intervals for the previous experiments as the high number of samples results in solutions with extremely low variance.)

\begin{figure}[h!]
\centering

\begin{subfigure}{0.32\textwidth}\centering
  \includegraphics[width=\linewidth]{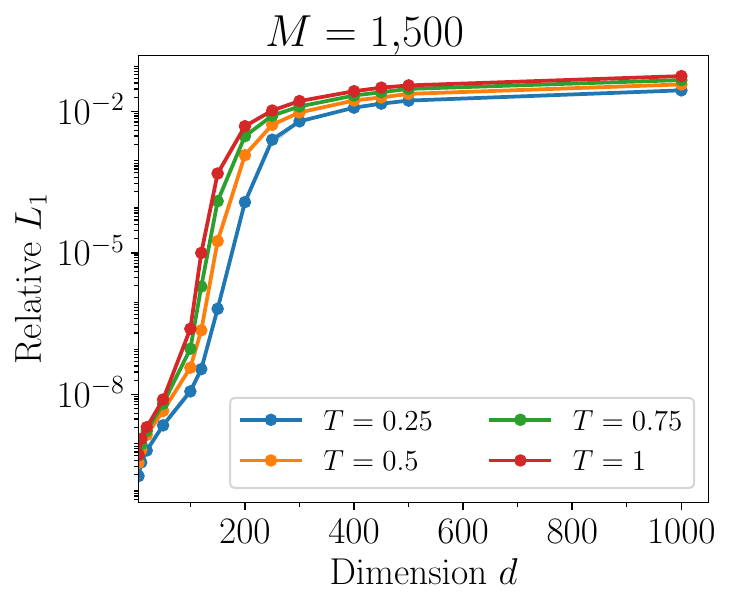}
  \caption{Rel.\ $L_1 = \frac{\|\tilde u_{\mathrm{IS}}-u\|_1}{\|u\|_1}$}
\end{subfigure}
\begin{subfigure}{0.32\textwidth}\centering
  \includegraphics[width=\linewidth]{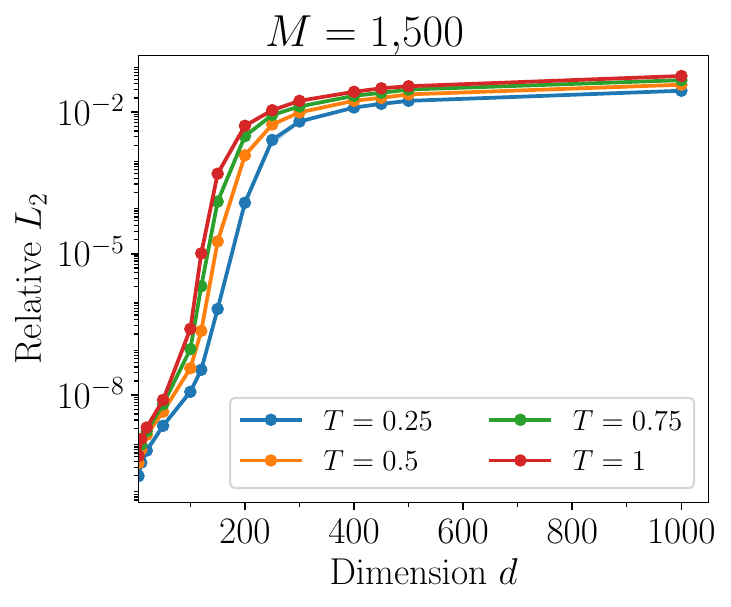}
  \caption{Rel.\ $L_2 = \frac{\|\tilde u_{\mathrm{IS}}-u\|_2}{\|u\|_2}$}
\end{subfigure}
\begin{subfigure}{0.32\textwidth}\centering
  \includegraphics[width=\linewidth]{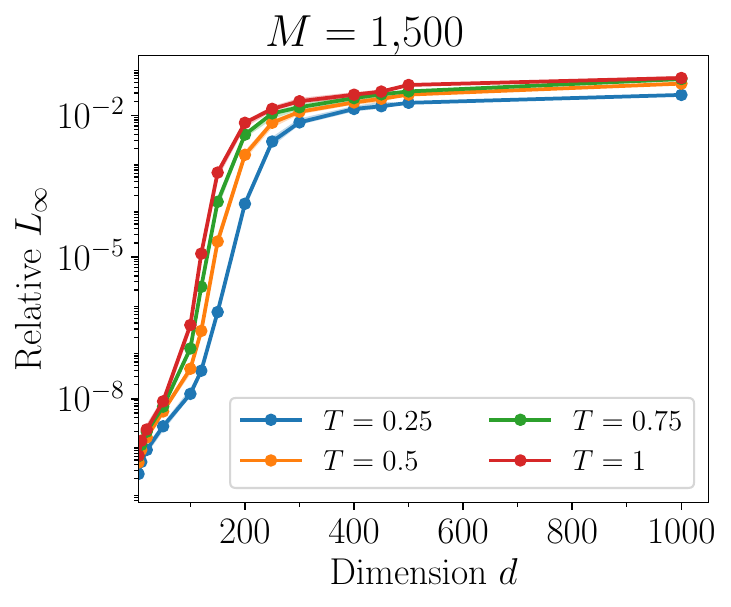}
  \caption{Rel.\ $L_\infty = \frac{\|\tilde u_{\mathrm{IS}}-u\|_\infty}{\|u\|_\infty}$}
\end{subfigure}\hfill
\includegraphics[width=0.32\textwidth]{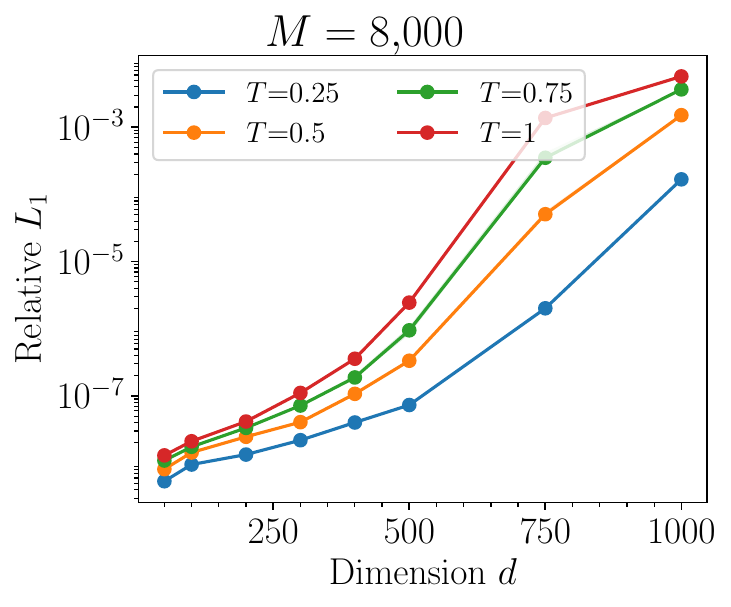}
\includegraphics[width=0.32\textwidth]{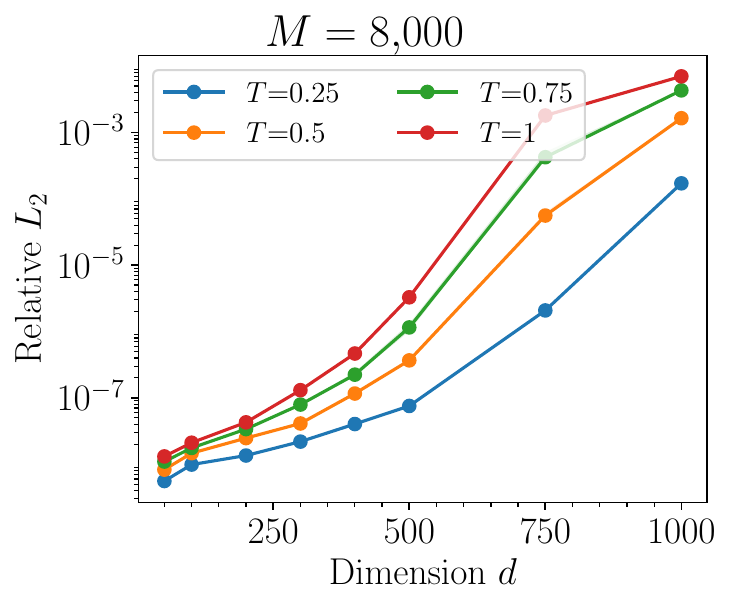}  \includegraphics[width=0.32\textwidth]{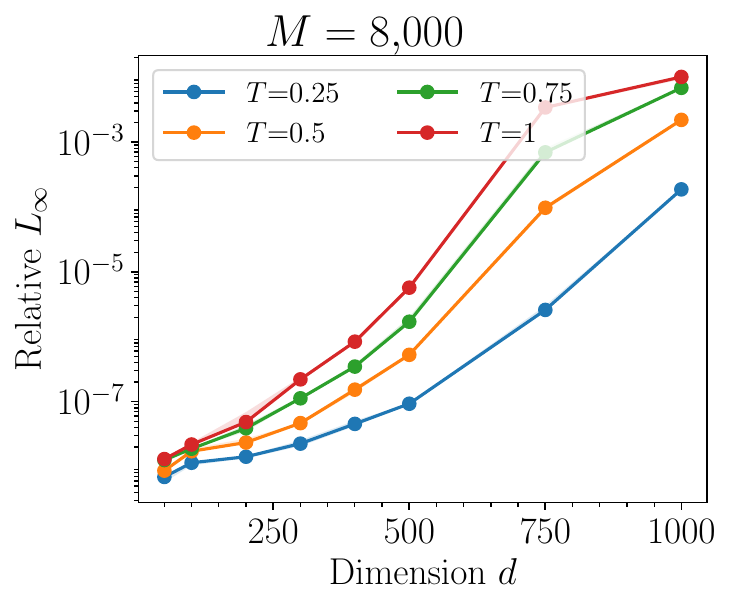}  

% Row 3 (centered)
\caption{First row: Error metrics for the HEATNETs solution of \eqref{example-high-d} for various time horizons, using $N_{PDE} =10,000, N_{IC} =2,048$ training points. First row: HEATNET with $M=1,500$ neurons (features) ($M_0 = 500, M_1 = 1,000$). Second row: HEATNET with $M = 8,000$ neurons (features) ($M_0 = 3,000, M_1 = 5,000$). }
\label{fig:high-dim-PINN-graphs}
\end{figure}

\begin{table}[ht]
\centering
\begingroup
\small
\setlength{\tabcolsep}{3.5pt}%
\renewcommand{\arraystretch}{1.0}%
\begin{tabular}{@{}l*{4}{cc}@{}}
\toprule
& \multicolumn{2}{c}{\(d=100\)}
& \multicolumn{2}{c}{\(d=500\)}
& \multicolumn{2}{c}{\(d=1000\)}
& \multicolumn{2}{c}{\(d=2000\)} \\
\cmidrule(lr){2-3}\cmidrule(lr){4-5}\cmidrule(lr){6-7}\cmidrule(lr){8-9}
Metric
& \(M=8,000\) & \(M=10,000\)
& \(M=8,000\) & \(M=10,000\)
& \(M=8,000\) & \(M=10,000\)
& \(M=8,000\) & \(M=10,000\) \\
\midrule
\(\mathrm{Rel.}\,L_1\)   & 9.35E-08 & 9.41E$-08$ & 1.61E-07 & 1.39E-07 & 8.73E-04 & 8.39E-05 & 8.94E-03 & 6.21E-03 \\
\(\mathrm{Rel.}\,L_2\)   & 9.94E-08 & 9.99E$-08$ & 1.65E-07 & 1.42E-07 & 9.38E-04 & 9.20E-05 & 9.06E-03 & 6.43E-03 \\
\(\mathrm{Rel.}\,L_{\infty}\)  & 1.29E-07  & 1.58E$-07$ & 1.55E-07 & 2.04E-07 & 1.35E-03 & 1.49E-04 & 1.06E-02 & 7.29E-03 \\
 Time (min) & 5.16 & 6.49 & 26.03 & 38.23 & 113.18 & 115.13 & 348.89 & 437.18 \\
\bottomrule
\end{tabular}
\endgroup
\caption{Relative error metrics and model building and training time for various dimensions $d$ and selection of HEATNET features $M$, with constant training samples $N_{PDE}= 15,000$ and $N_{IC} =3,000$ and time horizon $T = 0.5$.}\label{tab:error_metrics_by_d_M}
\end{table}

\begin{figure}[h!]
\centering

\includegraphics[width=0.32\textwidth]{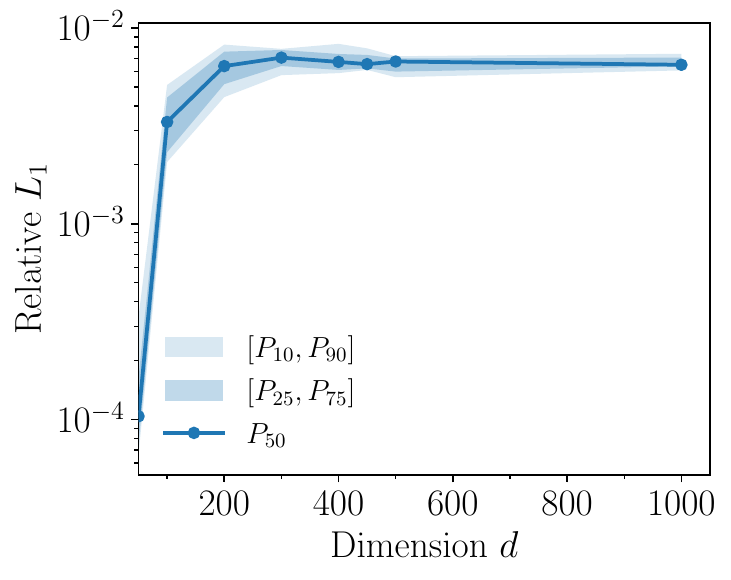}\hfill
\includegraphics[width=0.32\textwidth]{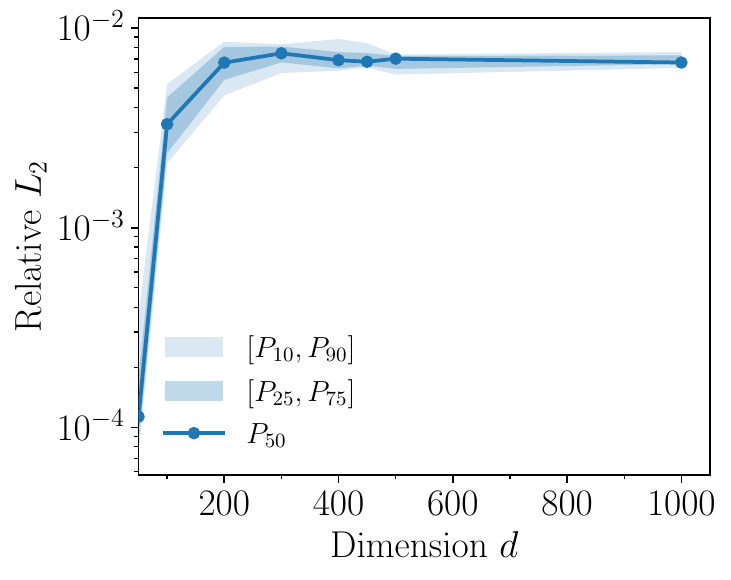}  \includegraphics[width=0.32\textwidth]{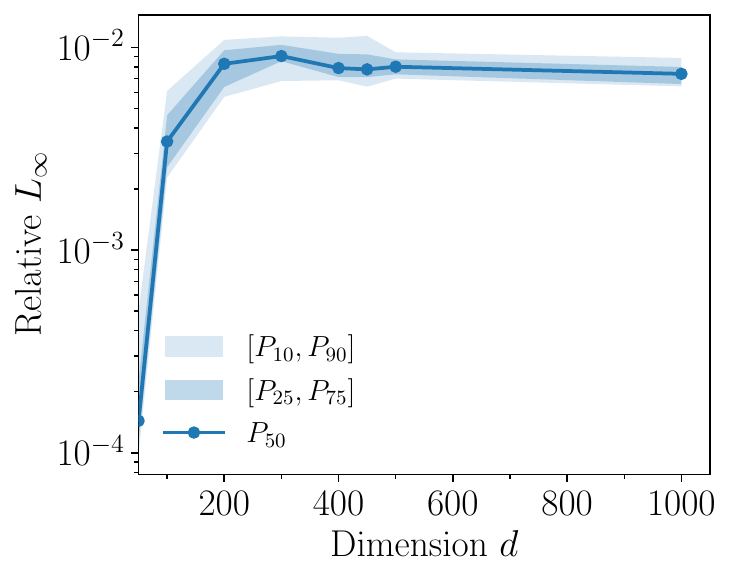}  

% Row 3 (centered)
\caption{Median error metrics, $P_{50}$, 10/90 percentile bands, $[P_{10}, P_{90}]$, and interquartile range, $[P_{25}, P_{75}]$, for the HEATNETs solution of \eqref{example-high-d}, for a time horizon $T =0.5$, using 600 training points ($N_{PDE} = 500, N_{IC} =100$) and $M=8,000$ neurons (features) in the hidden layer ($M_0 = 3,000, M_1 = 5,000$) with Sobol sampling. (Error metrics defined as in Fig. \ref{fig:high-dim-PINN-graphs}.)  }
\label{fig:high-dim-PINN-graphs-under}
\end{figure}

The above experiments provide a holistic view of the performance of the proposed HEATNETs approach. From Fig. \ref{fig:high-dim-PINN-graphs} and Table \ref{tab:error_metrics_by_d_M} we see that HEATNETs achieve exceptionally high accuracy, with relative $L_2$ errors as low as E-08 - E-07 for dimensions up to $d = 500$, and E-05 - E-04 when reaching $d = 1,000$ by increasing training data and features. The tests in the underdetermined regime show that HEATNETs exhibit excellent robustness; the relative $L_2$ error begins at approximately E-05 for $d=50$, rising to approximately E-03 at $d=100$ and remaining stable up to $d=1,000$.  Such a ``plateau'' has been also observed in other relevant works (see detailed comments and comparisons below). 

\subsection{Benchmark Example 3: PDE up to $d = 1,000$}
Here, for our illustrations, we adjusted \eqref{example-high-d} to obtain the PDE:
\begin{equation}\label{example-non-sep}
\begin{aligned}
u_t(t,x) = D \Delta u(t,x) + F(t,x), \,\, \text{with} \\
q(t) = \alpha t^2 \exp{(-t/3)} \,\,\, \text{and }\,\,
E(t,x) = \exp{\left(-\beta t \frac{\| x\|^2}{d}\right)},
\end{aligned}
\end{equation}
where the forcing term is constructed as below:
\begin{flalign}
        F(t,x) = \big(g'(t)+ & D k^2 g(t)\big)S_k(x) + \Big(q'(t)-\beta\,\tfrac{\|x\|^2}{d}q(t)\Big)E(t,x)S_m(x) \notag \\ & - Dq(t)\Big[E(t,x)\,\big(-m^2 S_m(x)\big)+S_m(x)\Delta E(t,x)+2\,\nabla E(t,x)\!\cdot\!\nabla S_m(x)\Big].
\end{flalign}
The closed form solution is $u(t,x) =  g(t)\,S_k(x) + q(t) E(t,x)S_m(x)$, where we used $(k,m) = (2,3)$.

\par For the solution of the PDE \eqref{example-high-d}, we used $M = 10,000$ features ($M_0 = 4,000, M_1 = 6,000$), $N_{PDE} = 15,000$ and $N_{IC} = 3,000$ samples for model training. As above, we use $\lambda_{IC} = \sqrt{5}$ and $\lambda_{\mathrm{ridge}} = 1.0$E$-06$ to solve the least squares problem. We consider time horizons $T = 0.5,1.0$. The results are gathered in Fig. \ref{fig:high-dim-non-sep}. As seen, the HEATNET is again able to achieve very accurate results, never exceeding relative errors of the order $5.0$E$-03$. We note that the computational time is of the same order as the results shown in Table \ref{tab:error_metrics_by_d_M}, when using $M = 10,000$ features. (We note that, as above, we can extend consider even higher dimensions by increasing the features and/or computational time).   

\begin{figure}[h!]
\centering
\includegraphics[width=0.32\textwidth]{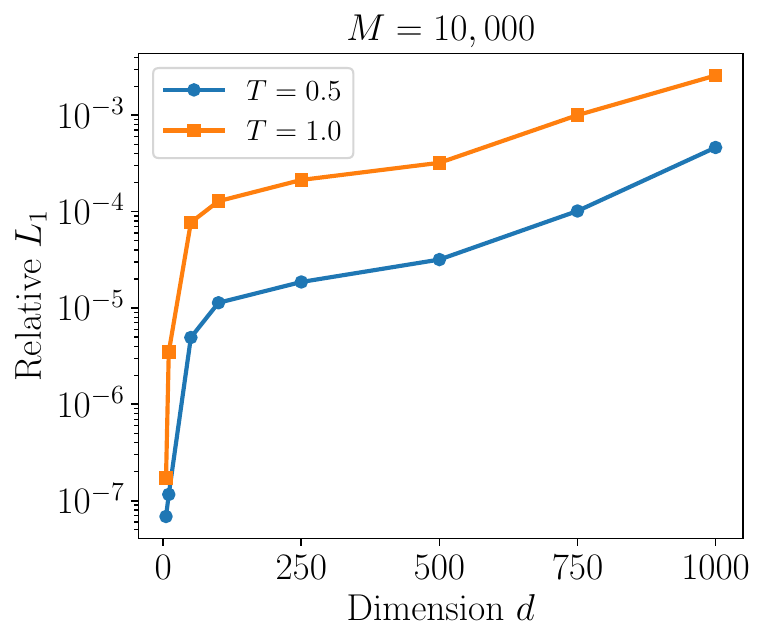}
\includegraphics[width=0.32\textwidth]{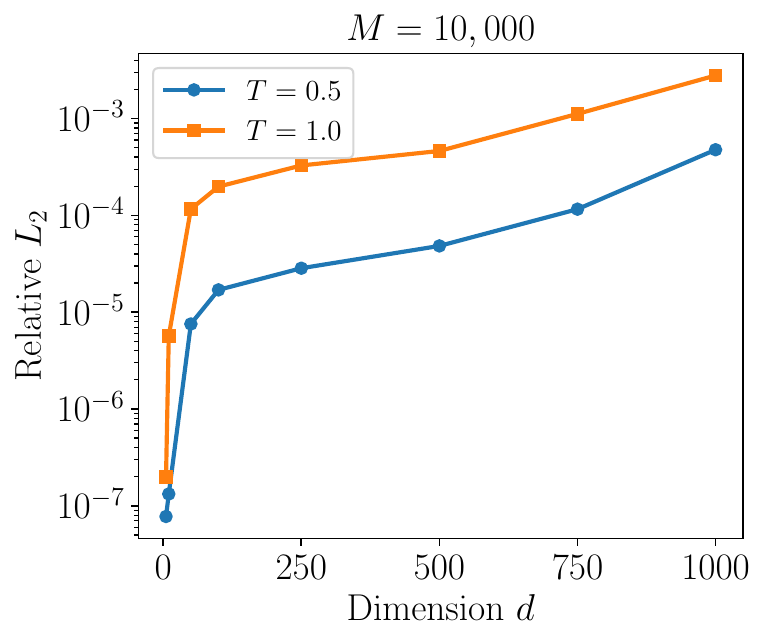}  \includegraphics[width=0.32\textwidth]{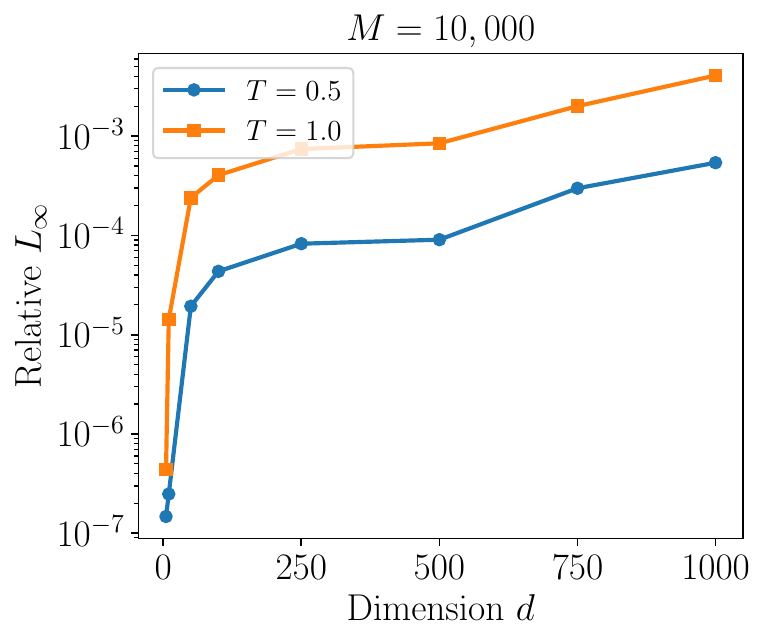}  
% Row 3 (centered)
\caption{HEATNET results for PDE \eqref{fig:high-dim-non-sep}, using $M = 10,000$ features and $N_{PDE} = 15,000, N_{IC} = 3,000$.}
\label{fig:high-dim-non-sep}
\end{figure}

The examples above have shown that, in relatively lower dimensions until $d=100$, our approach achieves higher accuracy compared to other approaches \cite{gaby2024neural, cho2023separable, xu2025deep,wang2024extreme} for the solution of similar problems. 

When moving to higher dimensions $d>100$, our approach still improves on the current literature; e.g., in \cite{menon2025anant}, the linear Poisson equation in $d=300$ is solved achieving a relative $L_2$ error of $6.00$E-02. In comparison, our results show relative $L_2$ of the order E-07, for $d=300$ for the diffusion PDE \eqref{example-high-d} (as seen in Fig. \ref{fig:high-dim-PINN-graphs}) and E-05 for \eqref{example-non-sep} (Fig. \ref{fig:high-dim-non-sep}).  
HEATNETs achieve relative $L_2$ errors of the order E-08 for $d = 100$ and E-05 - E-04 for $d=1,000$ when solving PDE \eqref{example-high-d} and E-05 and E-04 - E-03, respectively, when solving PDE \eqref{example-non-sep}. 
As shown in Fig. \ref{fig:high-dim-PINN-graphs} and Fig. \ref{fig:high-dim-PINN-graphs-under} the error metrics exhibit a plateau which, as mentioned above, is in line with results reported in other studies (see e.g., in \cite{hu2024tackling}, where such a plateau is also observed for dimensions ranging from $d=1,000$ to $d = 100,000$). 

\section{Discussion/ Conclusions}

In this work, we introduced HEATNETs: tailor-made ``functional analysis-informed'' random feature neural networks for the solution of parabolic PDEs in high dimensions, based on the mild representation of the solution. We prove that HEATNETs provide universal approximators for parabolic PDE solutions in arbitrary dimensions and attain a sample-complexity-type convergence of the order of $\mathcal{O}(N^{-1/2})$. Furthermore, we provide a detailed mathematical analysis regarding the numerical implementation of the model and its features, by handling the singularities that arise in the Gaussian kernels and using importance sampling to move to high dimensions.

\par To provide a thorough overview of the construction, training and applicability of HEATNETs, we have used various benchmark linear problems in dimensions from 1 to 2,000 for which analytical solutions are available. As shown, the scheme achieves a remarkable numerical approximation accuracy of the order of $\sim$E-08/E-07 for up to 200/500 dimensions, and of the order of $\sim$E-03/E-04 for dimensions from 1,000 to 2,000, thus achieving higher accuracy compared to other approaches for similar problems. 

Importantly, HEATNETs constitute an explainable machine learning approach with relatively low model complexity/ dimensionality. Indeed, in our examples, we show that with less than 100 total features we are able to achieve extremely high accuracy for 1D PDEs, and considering only 1,500 nodes (features) suffices for relative $L_1, L_2$ and $L_\infty$ errors of the order of E-08 - E-07 for PDE \eqref{example-high-d}, for dimensions up until $d = 100$. Finally, with nodes (features) from 8,000 - 10,000 we are able to achieve errors of the order E-04 - E-03 (for the studied metrics) for dimensions $d = 1,000 - 2,000$ across all studied examples. These results can be (loosely/tentatively) compared to prior results of other approaches for the solution of linear PDEs. Indicatively, in \cite{xu2025deep} a neural network with 2 hidden layers of 64 neurons is used to solve the 1D Laplace equation, amounting to approximately 4,000 trainable parameters, and this increased significantly when applying the method to the Black-Scholes PDE in $d=100$, where 4 hidden layers of 512 neurons are used (approximately 840,000 parameters).
In \cite{hu2024tackling} an MLP with 4 hidden layers of 128 neurons is used for the linear Poisson PDE (e.g., for $d=5$ this leads to over 50,000 trainable parameters). In \cite{wang2024extreme} a single layer extreme learning machine was used with 2,000 nodes for the linear Poisson PDE up to $ d= 15$ and 3000 nodes were used for the linearized  Korteweg–De Vries equation up to $ d =10$. Other approaches like \cite{cai2025deep,gaby2024neural,cho2023separable} combine standard fully-connected neural networks with additional underlying structures that increase the total number of parameters. Compared (loosely) to prior results, HEATNETs therefore provide a more interpretable and explainable alternative. 

Regarding the computational requirements, other methods such as \cite{hu2024tackling,cho2023separable, cai2025deep} employ similar resources to this work, i.e., A100 machines or T4 GPUs (\cite{menon2025anant}), (\cite{xu2025deep} use an A40 GPU). 
For example, the HEATNET with $M = 1,500$ for e.g., $d=300$ requires around 900 seconds for building and training, for all problems considered.
Illustratively, for similar (linear) problems of the same dimension \cite{menon2025anant} report 127 minutes. In \cite{wang2024extreme} training requires 14.6 seconds for $d= 9$; the HEATNET with 2,000 features for PDE \eqref{example-high-d} and $d = 9$ similarly requires 10.9 seconds for training (and 9.3 seconds for building the $A,b$). Overall, HEATNETs provide a flexible framework to consider various combinations of the model hyperparameters and computational requirements needed to achieve specific levels of accuracy. (For relatively low dimensional cases, one can construct and train the model on a standard CPU).
While direct one-to-one comparisons are not always straightforward due to differences in PDEs, and implementation details, the experimental evidence nonetheless demonstrates clear advantages of our approach.  
A detailed comparison is beyond the scope of this paper and is left for future work.

The strong performance across experiments suggests that HEATNETs are able to avoid overfitting and generalize well. This is a critical feature for successfully mitigating the curse of dimensionality and is particularly evident from the under-determined setting (see Fig. \ref{fig:high-dim-PINN-graphs-under}), where we consider only $600$ collocation points for the PINN training approach and are still able to consistently reach errors of the order E-03 up to $d=1,000$ dimensions.  We attribute the performance of the HEATNETs to the tailor-made nature of the model, since the basis of the RFNN is constructed to fit the underlying mathematical theory of the mild solution to the parabolic PDE. \par 

To conclude, our primary scope is to introduce the method and its theoretical foundations; although it can be used to solve problems in arbitrarily higher dimensions (beyond the 2,000 demonstrated here), computational cost and memory requirements grow rapidly with the number of collocation points and dimensionality, so we leave addressing scalability for such very high-dimensional systems to future work. Extensions to more challenging problems, including non-linear PDEs,  will be also addressed in subsequent work. This extension is straightforward, using appropriate fixed-point iterations, that can be emulated using the concept of Fredholm Neural Networks, which we have recently introduced in \cite{georgiou2025fredholm1}.

\section*{Acknowledgments}

K.G. acknowledges support from the PNRR MUR Italy, project PE0000013-Future Artificial Intelligence Research-FAIR. 
C.S. acknowledges partial support from the PNRR MUR Italy, projects PE0000013-Future Artificial Intelligence Research-FAIR \& CN0000013 CN HPC - National Centre for HPC, Big Data and Quantum Computing. Also from the Istituto di Scienze e Tecnologie per l'Energia e la Mobilità Sostenibili (STEMS)-CNR. A.N.Y. acknowledges the use of resources from the Stochastic Modelling and 
Applications Laboratory, AUEB.  

\bibliographystyle{plain}
%\section*{References}
\bibliography{mybibfile}
%\section*{References}

\end{document}